\documentclass[aop]{imsart2}
\usepackage{neuralnetwork}
\usepackage{amsthm, amssymb, amsmath,verbatim}
\usepackage{enumerate}
\usepackage{amsmath}
 \usepackage{verbatim}
\usepackage{hyperref}

\usepackage{tikz}
\usetikzlibrary{calc}

\usepackage{appendix}

\usepackage{mathptmx}
\usepackage[T1]{fontenc} 
\usepackage[utf8]{inputenc}
\usepackage{lmodern}

\theoremstyle{plain}

\newtheorem{theorem}{Theorem}[section]
\newtheorem{proposition}[theorem]{Proposition}
\newtheorem{lemma}[theorem]{Lemma}
\theoremstyle{remark}

\newtheorem{remark}[theorem]{Remark}

\def \p {\mathbf p}
\def \n {\mathbf n}
\def \q {\mathbf q}
\def \cB {\mathcal B}

\def \cP {\mathcal P}
\def \cK {\mathcal K}
\def \cS {\mathcal S}

\def \bV {\mathbb V}
\def \bZ {\mathbb Z}
\def \bR {\mathbb R}
\def \bN {\mathbb N}
\def \bP {\mathbb P}
\def \bQ {\mathbb Q}
\def \rmd {\mathrm d}
\def \rmB {\mathrm B}
\def \rme {\mathrm e}

\newcommand{\R}{\mathbb{R}}

\newcommand{\N}{\mathbb{N}}

\begin{document}
\begin{frontmatter}
\title{Large deviations  at the origin of random walk in random environment}

\runtitle {Large deviation at the origin of RWRE}

\begin{aug}
  \author[A]{\fnms{Alexander}
    \snm{Drewitz}\ead[label=e2,mark]{adrewitz@uni-koeln.de}},
  \author[B,C]{\fnms{Alejandro F.} \snm{Ram\'irez}\ead[label=e1,mark]{ar23@nyu.edu}}, 
   \author[C]{\fnms{Santiago} \snm{Saglietti}\ead[label=e4,mark]{sasaglietti@mat.uc.cl}}
\and 
\author[D]{\fnms{Zhicheng} \snm{Zheng}\ead[label=e3,mark]{zz4230@nyu.edu}}

\address[A]{Department of Mathematics and Computer Sciences, 
  Universit\"at zu K\"oln
  \printead{e2}}

\address[B]{NYU-ECNU Institute of Mathematical Sciences, 
NYU Shanghai 
\printead{e1}}

\address[C]{Facultad de Matem\'aticas,
Pontificia Universidad Cat\'olica de Chile 
\printead{e4}}

\address[D]{Courant Institute of Mathematical Sciences, 
  New York University 
  \printead{e3}}

\end{aug}

\begin{abstract}
  We consider a random walk in an i.i.d.\ random environment on~$\mathbb Z^d$
  and study properties of its large deviation rate function at the origin. It was proved 
  by Comets, Gantert and Zeitouni in dimension $d=1$ in 1999 and later by Varadhan in dimensions $d\ge 2$ in 2003 that, for uniformly elliptic i.i.d.\ random environments,  the quenched and the averaged large deviation
  rate functions coincide at the origin. Here we provide a description of an atypical event realizing the correct quenched large deviation rate in the nestling and marginally nestling setting: the random walk seeks regions of space where the environment emulates the element in the convex hull of the support of the law of the environment at a site which minimizes
  the rate function. Periodic environments play a natural role in this description.
\end{abstract}



\end{frontmatter}

\section{Introduction}
The model of random walk in random environment, describing the motion of a particle in a highly disordered medium,
has been an important cornerstone in the microscopic description of homogenization phenomena \cite{R11} and has more recently been
shown to fall within the KPZ universality class \cite{BC17}. Several fundamental questions about this model have not been fully answered yet, such as,
for example, the validity of the law of large numbers \cite{DR14} or the relation between directional transience and ballisticity.
Some understanding has, however,  been achieved with regards to the large deviations of the random walk, both at the quenched
and the annealed level. We refer to \cite{GdH94} for the first proof of the quenched large deviation principle for i.i.d.\ environments in dimension $d=1$, \cite{CGZ00} for an extension of the quenched large deviation principle to ergodic environments and the first proof of the annealed large deviation principle in dimension $d=1$, and \cite{Z98} for the quenched large deviation principle in dimensions $d\ge 1$ for the nestling case. The general proof of both a quenched and an annealed large deviation principle for uniformly elliptic i.i.d.\ random environments was established by
Varadhan \cite{V03} in 2003 for all dimensions $d\ge 1$.
In \cite{CGZ00} and \cite{V03}, it was also shown that at the origin both, the annealed and quenched large deviation rate functions coincide, and furthermore, 
a variational formula for their common value was provided. 

In this article, we describe a dominant behavior giving rise to the correct rate of exponential decay as $n \to \infty$ of the large deviation probability that the random walk returns to
the origin after $n$ steps in the so-called nestling or marginally nestling cases. In the quenched setting, we prove the following properties of this behavior, which we call
the {\it dominant behavior event} or {\it dominant event}, and of the rate function at $0$: ($i$) the rate
function at $0$ is equal to the rate function at 0 of a random walk in a certain periodic environment, which we shall refer to as the
{\it optimal environment}, whose site marginals are 
all identical  and correspond to a point in the boundary of the convex hull of the support of the law of the  random environment; ($ii$) on this dominant event, the random walk moves quickly to a region where the
environment is close to a periodic environment whose rate function coincides with the rate function of the random walk in the optimal environment; ($iii$) this periodic environment uses jump probabilities which are close to jump probabilities of the support of the law of the random environment. These results
complement previous ones, giving insight into the mechanism under which atypical behavior of the random walk occurs. In particular, they complement what is known in the nestling case, where the random walk looks for a trap or nest (see Zerner \cite{Z98} for the introduction of this terminology and Sznitman \cite{S00} for further developments), and stays there for most of the time. A homogenization result of a viscous Hamilton-Jacobi equation in a periodic potential is implicit in our results, including a representation of the homogenized Hamiltonian at $0$ (see \cite{LPV87} and \cite{E89} for the proof of homogenization of the viscous Hamilton-Jacobi equation in a periodic potential).

To provide some context about these questions, we now describe  some of the main results about large deviations of random walk
in random environment. In 1997, Greven and den Hollander \cite{GdH94} proved a quenched large deviation principle for random
walks in a uniformly elliptic i.i.d.\ random environment in dimension $d=1$. In 1999, Comets, Gantert and Zeitouni \cite{CGZ00}
extended this quenched large deviation principle to ergodic environments and, in addition, proved an annealed large
deviation principle for i.i.d.\ uniformly elliptic environments in dimension $d=1$. Zerner in  2000 \cite{Z98} proved a multidimensional
large deviation principle for the so-called random walks in nestling
environment. Finally, in 2003, Varadhan \cite{V03} proved a quenched
large deviation principle in dimensions $d\ge 1$ for random walk in a uniformly elliptic ergodic environment, together with
an annealed large deviation principle for random walk in uniformly elliptic i.i.d.\ environment. Furthermore, there he also showed that,
for any dimension $d\ge 1$, the quenched and the averaged rate
functions at $0$ are equal. Subsequently, Yilmaz in \cite{Y11} proved
that in dimensions $d \ge 4$, whenever the so called ballisticity condition $(T)$ (cf.\ \cite{S01,BDR11}) is satisfied,
there is a neighborhood of the velocity where the quenched and
annealed large
deviation rate functions are equal. As a complementary result, in
\cite{BMRS23} and \cite{BMRS23b}, it
was proven that in dimensions $d\ge 4$, given any compact set in the
interior or in the boundary of the domain of the rate functions not containing $0$, if the disorder or
randomness of the environment is low enough, the quenched and annealed
large deviation rate functions coincide in this compact set.  

The structure of the paper is as follows. 
We will first introduce in Section \ref{sec:NotRes} the notations and state the main results. Then we continue  by investigating large deviations for a random walk in time-periodic environment so that it necessarily spends prescribed fractions of the time in certain environments in Section \ref{sec:LDRWTP}. This analysis is comparatively easy but serves as a main building block for our more complex setup of space-periodic environments  investigated in Section \ref{sec:LDRWSP}. Indeed, space-periodic environments will play an important role since their mesoscopic approximations can naturally be found in large spatial balls due to a Borel-Cantelli type argument and our independence assumption. Using a refined analysis  we can then -- again approximately -- reduce the random walk in certain space-periodic environments to the one investigated in  Section \ref{sec:LDRWTP}. This, among other tools such as local central limit theorems for RWRE, will be a key ingredient in the proof of the main result which is conducted in Section \ref{sec:proofMain}.
\section{Notation and results} \label{sec:NotRes} To state our results, we begin by defining the models of random walk in random environment and random
walk in a periodic environment (or RWRE and RWPE, for short) on $\bZ^d$. 

For $p \in [1,\infty],$ denote by $|x|_p$ the $\ell^p$-norm of any $x \in \bR^d$ and introduce the set
$\bV:=\{e\in\mathbb Z^d: |e|_1=1\}$ of all unit vectors in $\bZ^d$. Furthermore, we define 
$$\mathcal P_\bV:=\Big\{\vec{p}=(p(e))_{e\in
  \bV} \in [0,1]^\bV : \sum_{e\in \bV}p(e)=1\Big\},$$
  the simplex of all probability vectors 	on $\bV$, and the product space $\Omega:=(\mathcal {P_\bV})^{\mathbb Z^d}$ endowed with the  product topology. Any element $\omega \in \Omega$ will be called an \textit{environment}, i.e., each $\omega=(\omega(x))_{x \in \bZ^d}$ is a family of probability vectors $\omega(x)=(\omega(x,e))_{e \in \bV}$ on $\bV$, indexed by the sites $x \in \bZ^d$.
  
Now, given some $x \in \bZ^d$ and $\omega \in \Omega$, we define the \textit{random walk in the environment $\omega$ starting from $x$} 
as the discrete time Markov chain $(X_n)_{n \in \bN_0}$ on $\bZ^d$ with law $P_{x,\omega}$ given by $P_{x,\omega}(X_0=x)=1$ and, for each $n \in \bN_0$, $y\in \bZ^d$ and $e \in \bV$, 
\[
P_{x,\omega}(X_{n+1}=y+e|X_n=y)=\omega(y,e).
\]
 Now, if the environment $\omega$ is chosen at random according to some probability measure $\bP$ on $\Omega$, we obtain a probability~law $P_x$ on $(\bZ^d)^{\bN_0} \times \Omega$ 
 (endowed with the product $\cB((\bZ^d)^{\bN_0}) \otimes \cB(\Omega)$  of the underlying Borel $\sigma$-algebras) 
 given by their semi-direct product, i.e.\ for any $A \in \cB((\bZ^d)^{\bN_0})$ and $B \in \cB(\Omega)$,
\[
P_x (A \times B) = \int_B P_{x,\omega}(A) \, \rmd \bP(\omega).
\]
We will refer to $P_{x,\omega}$ as the \textit{quenched law} 
and $P_x$ is the {\it averaged}
(or {\it annealed}) law of the {\it random walk in random environment} (RWRE), and we call $\bP$ the \textit{environment law} of the RWRE. Throughout the sequel, we will say that the random environment is i.i.d.\ if the random probability vectors $(\omega(x))_{x \in \bZ^d}$ are i.i.d.\ under $\bP$ and, moreover, say it is \textit{uniformly elliptic} if there exists $\kappa > 0$ such that
\[
\bP(\omega(x,e)\ge\kappa)=1 \quad \text{for all $x\in\mathbb Z^d$ and $e \in \bV$},
\] i.e., if there exists some $\kappa > 0$ such that the environment law $\bP$ has marginals which are supported on the set
\begin{equation} \label{eq:Pkappa}
\cP_{\bV}^{(\kappa)}:=\big \{ \vec{p} \in \cP_\bV : p(e) \geq \kappa \text{ for all }e \in \bV \big\}.
\end{equation}

Finally, in order to introduce periodic environments, fix some ${\bf n}=(n_1,\ldots,n_d) \in \bN^d$ and define the box $\mathbf{B}_{\n}:=\{x\in\mathbb Z^d: 0\le x_i\le n_i -1\ {\rm for}\  1\le i\le d\}$. In addition, for $x\in \bZ^d$ set
\[
x\ {\rm mod}\ {\bf n}=(x_1 \ {\rm mod} \ n_1,\ldots, x_d \ {\rm mod} \ n_d) \in \mathbf{B}_\n,
\] and, given ${\bf p}=(\vec{p}(x))_{x\in \mathbf{B}_{\bf n}} \in (\cP_\bV)^{\mathbf{B}_\n}$, define the environment $\omega_{\n,\p}\in\Omega$ by the formula
\begin{equation}
  \label{erwpe}
\omega_{\n,\p}(x)=\vec{p}(x\ {\rm mod} \ {\bf n})\quad \text{ for each }x\in \mathbb Z^d.
\end{equation}
Now, given $x \in \bZ^d$, the {\it random walk in the periodic environment (RWPE) $\omega_{\n,\p}$ starting from $x$} is the random walk on $\bZ^d$ with law $P_{x, \omega_{\n,\p}}$. We call the pair $(\n,\p)$ the \textit{period} of $\omega_{\n,\p}$.
 
Both RWRE and RWPE satisfy large deviation principles, which we describe next. For this purpose, let  
\[
\mathrm{D}:=\{x\in\mathbb R^d:|x|_1\le 1\} 
\] 
denote the unit $\ell^1$ unit ball in $\R^d$. 
In \cite[Theorems 2.3 and 3.2]{V03}, Varadhan established both a quenched and an averaged large deviation principle for random walks in i.i.d.\ (actually, for ergodic environments in the quenched case) uniformly elliptic 
environments on $\bZ^d$ for any dimension $d\ge 1$:

\begin{itemize}

\item[(i)] There exists a rate function $I_a:\mathrm{D}\to [0,\infty)$, which is both convex and continuous on the interior $\mathrm{D}^\circ$ of ${\rm D}$, such that, 
  for every Borel set $A \subseteq D$,
 \[
  -\inf_{x\in A^\circ}I_a(x)\le \liminf_{n\to\infty}\frac{1}{n}\log P_0\bigg(\frac{X_n}{n}\in A\bigg) 
  \le  \limsup_{n\to\infty}\frac{1}{n}\log P_0\bigg(\frac{X_n}{n}\in A\bigg)\le -\inf_{x\in\bar A}I_a(x). 
  \]
  
\item[(ii)] There exists a deterministic rate function $I_q:\mathrm{D}\to [0,\infty)$, convex and continuous on $\mathrm{D}^\circ$, such that, for $\mathbb P$-almost every $\omega$, given any Borel set $A \subseteq D$ one has 
  \[
  -\inf_{x\in A^\circ}I_q(x)\le \liminf_{n\to\infty}\frac{1}{n}\log P_{0,\omega}\bigg(\frac{X_n}{n}\in A\bigg) 
  \le  \limsup_{n\to\infty}\frac{1}{n}\log P_{0,\omega}\bigg(\frac{X_n}{n}\in A\bigg)\le -\inf_{x\in\bar A}I_q(x). 
  \]
\end{itemize}
In addition, in \cite{V03} it has been established that both rate functions agree at the origin, i.e.\ $I_a(0)=I_q(0)$. More precisely, given an arbitrary subset $S\subseteq \mathcal P_\bV$, let $\mathcal K_S$ denote its
convex hull and, given an environment law $\mathbb P$, define $\mathcal
K_{\mathbb P}:=\mathcal K_{\cS(\bP)}$ where $\cS(\bP)$ denotes the essential support of $\omega(0)$. Then, in~\cite{V03} the following was shown: 
\begin{itemize}
\item[(iii)] For any dimension $d\ge 1$, 
\begin{equation} \label{eq:rateFunc}
  I_a(0)=I_q(0)=-\log\Big(\inf_{\theta\in\mathbb R^d}\sup_{\vec{p} \in \mathcal K_{\mathbb P}} \sum_{e\in \bV}e^{\langle \theta, e\rangle }p(e)\Big).
\end{equation}
\end{itemize}

Furthermore, given a RWPE with law $P_{0,\omega_{\n,\p}}$, there exists a convex function $I_{\n,\p}$ such that for every Borel set $A \subseteq \bR^d$,
\begin{equation*}
\begin{split}
-\inf_{x\in A^\circ}I_{\n,\p} (x) &\le \liminf_{n\to\infty}\frac{1}{n}\log P_{0,\omega_{\n,\p}} \bigg(\frac{X_n}{n}\in A\bigg)\\ &\le 
  \limsup_{n\to\infty}\frac{1}{n}\log P_{0,\omega_{\n,\p}} \bigg(\frac{X_n}{n}\in A\bigg)\le -\inf_{x\in\bar A}I_{\n,\p}(x). 
\end{split}
\end{equation*}
Indeed, this is an immediate consequence of the quenched large deviation principle \cite[Theorem 2.3]{V03} applied to RWRE with the ergodic environment distribution 
\[
\bP_{\n,\p} := \frac{1}{|\mathbf{B}_{\n}|} \sum_{x \in \mathbf{B}_{\n}} \delta_{\omega_{\n,\p}(\cdot + x)}.
\]
Finally, to state the main result of this article, we need to introduce first the concept of nestling, marginally nestling and non-nestling RWRE. To this end, given any $\vec{p} \in \cP_{\bV}$, we define the \textit{drift} of $\vec{p}$ as
\begin{equation}
\label{eq:defdrift}
\rmd(\vec{p}):=\sum_{e\in \bV} p(e) e
\end{equation} and denote by $\mathcal{D}_0:=\{ \vec{p} \in \cP_\bV : \rmd(\vec{p})=0\}$ the set of all zero drift probability vectors on~$\bV$.
  Then, given a RWRE with environment law $\bP$, we will say that the RWRE is \textit{nestling} if $\text{int}(\mathcal{K}_\bP) \cap \mathcal{D}_0 \neq \emptyset$, {\it marginally nestling} if $\text{int}(\mathcal{K}_\bP) \cap \mathcal{D}_0 =\emptyset$ but $\partial \mathcal{K}_{\bP} \cap \mathcal{D}_0 \neq \emptyset$ and {\it non-nestling} if $\mathcal{K}_\bP \cap \mathcal{D}_0=\emptyset$. 
  
  For $\varepsilon,\delta>0$ chosen to be small enough later on, given 
  any period $(\n,\p) \in \N^d \times (\cP_\bV)^{\mathbf{B}_\n}$ and $N\ge 1$, 
  we define the event $G^{\varepsilon, \delta}_{N,\n,\p}:=\bigcup_{x_\n \in \mathbf{B}_\n} G^{\varepsilon, \delta}_{N,\n,\p,x_\n}$, where for $y\in\mathbb Z^d$,
  \begin{equation} \label{def:GDef}
G^{\varepsilon, \delta}_{N,\n,\p,y}:=\Big\{ \omega \in \Omega : 
      |\omega(x) -\omega_{\n,\p}(x)|_\infty \leq \varepsilon \text{ for all }x \in \rmB(y,\delta(\log N)^{1/d})\Big\},
  \end{equation}
  and we recall that  $\rmB(x,r)$ denotes the $\ell^1$-ball on $\bZ^d$ of 
  radius $r > 0$ centered at $x \in \bZ^d$. 
Furthermore, let us set $a_N:=\lfloor\frac{N}{\log N} \rfloor$ and, for $x_\n \in \mathbf{B}_\n$, define the following events: 
\begin{enumerate}
    \item [i.] $A_1^{(N,x_\n)}$, the event that the
random walk $(X_n)_{n \in \N_0}$ moves from $0$ to $x_\n$ in $a_N$ steps; 
\item [ii.] $A_2^{(N,x_\n)} = A_2^{(N,x_\n), \delta}$, the event that $X_{a_N}=x_\n$, $X_n \in \rmB(x_\n,\delta(\log N)^{1/d})$ for all $n \in [a_N,N-a_N]$ and $X_{N-a_N}=x_\n$;
\item [iii.] $A_3^{(N,x_\n)}$, the event that $X_{N-a_N}=x_\n$ and $X_N=0$. 
\end{enumerate} 
With this, define now the event $A^{\varepsilon, \delta}_{N,\n,\p}:=\bigcup_{x_\n \in \mathbf{B}_\n} A^{\varepsilon, \delta}_{N,\n,\p,x_\n}$, where
\[
A^{\varepsilon, \delta}_{N,\n,\p,x_\n}:= G^{\varepsilon, \delta}_{N,\n,\p,x_\n} \cap A_1^{(N,x_\n)} \cap A_2^{(N,x_\n)} \cap A_3^{(N,x_\n)},
\]
i.e., $A^{\varepsilon, \delta}_{N,\n,\p,x_\n}$ is the event that the random walk moves in $a_N=o(N)$ steps to a ball of
radius $\delta(\log N)^{1/d}$ (centered at $x_\n$) within  which the environment is at a distance $\varepsilon$ to the periodic environment $\omega_{\bf n, \bf p}$, stays there for a time that is strongly asymptotic to $N$, and then
returns to $0$ in $a_N$ steps (see Figure \ref{figure1sb}).

\begin{figure}
\begin{center} 
\vskip-3cm 
  \hspace*{-0.5cm}
  \includegraphics[height=15cm]{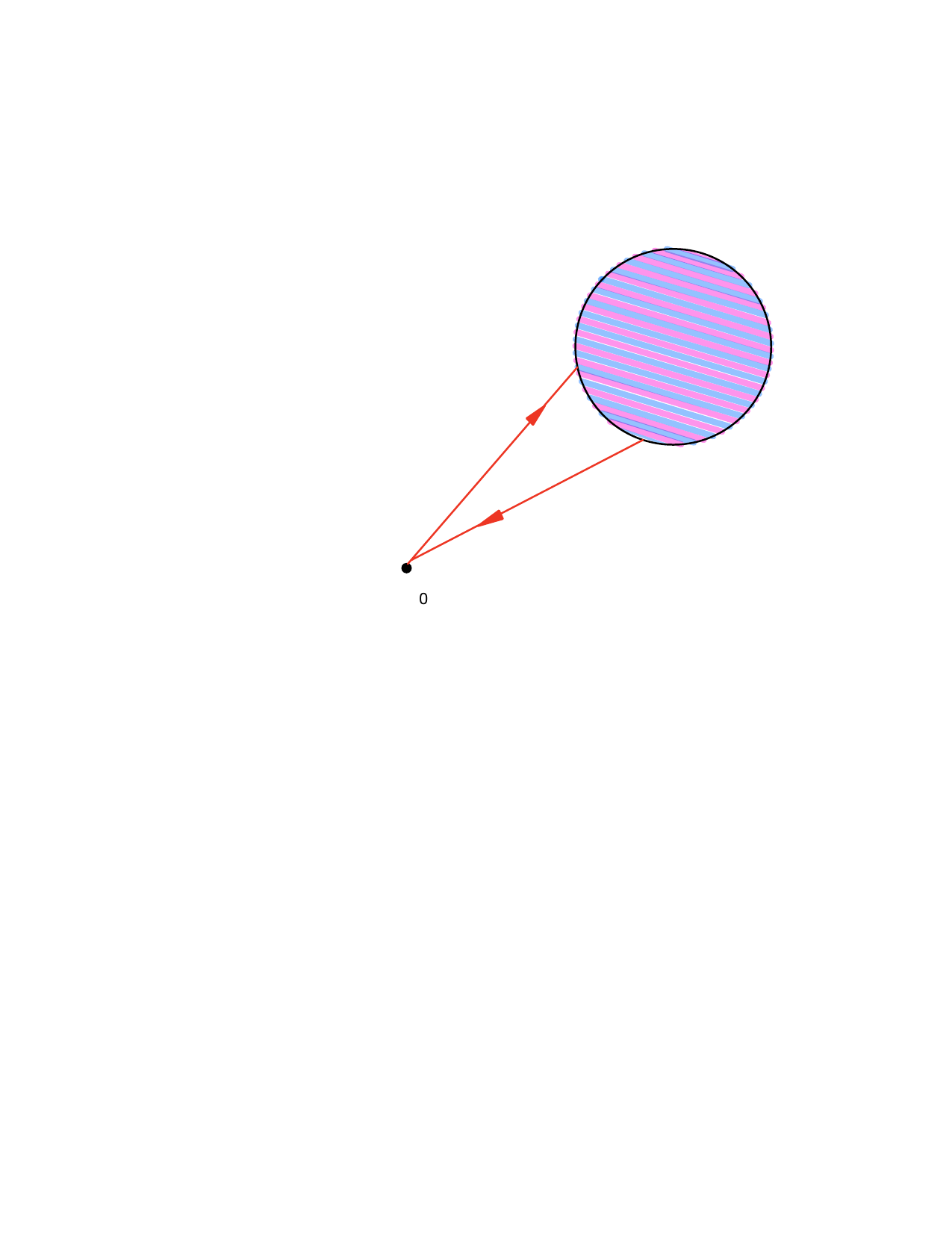}
 \vskip-7cm
\caption{Graphic representation of the event $A^\varepsilon_{N,\bf n,\bf p}$. The random walks starts at $0$, moving quickly through the upper red line to the ball, which is at a distance smaller than $N/\log N$, stays in the ball a time of order $N$ and then moves back through the second red line to the origin. The environment inside the ball, which has a radius $\delta(\log N)^{1/d}$, is periodic,  formed by alternating parallel strips each one composed of a different homogeneous environment belonging to the support of the law of $\omega(0)$.}
\label{figure1sb}
\end{center}
\end{figure}

We are now ready to state our main result.
  
\begin{theorem}
    \label{theorem1}
    Consider a random walk in a uniformly elliptic i.i.d.\ environment. 
Then, the following assertions hold true:
\begin{itemize}
\item[(i)] There exists $\vec{p}_\ast=(p_\ast(e))_{e \in \bV} \in \cK_{\bP}$
  such that  
\begin{equation}
\label{2p1}
I_q(0)=I_a(0)=-\log\Big(\sum_{e\in \bV}\sqrt{p_\ast(e)p_\ast(-e)}\Big)=:I(0).
\end{equation}
Furthermore, if the walk is either non-nestling or marginally nestling, then $p_*\in\partial 
\cK_{\mathbb P}$.

\item[(ii)] Suppose that the random walk is either non-nestling or marginally nestling. Then, for each $\varepsilon>0$ there exists some period $(\n_\varepsilon,\p_\varepsilon)$ such that
\[
|I(0)-I_{(\n_\varepsilon,\p_\varepsilon)}(0)|\le \varepsilon.
\]
\item[(iii)] Suppose that the random walk is either non-nestling or marginally nestling. Then, for each $\varepsilon>0$ there exists some period $(\n_\varepsilon,\p_\varepsilon)$ and a $\delta_0>0$ such that for every $\delta\in(0,\delta_0)$,
\begin{equation}
\label{bccl}
\mathbb P\big( \liminf_{N \to \infty} G^{\varepsilon, \delta}_{N,\bf n_\varepsilon,\bf p_\varepsilon}\big) =1 
\end{equation}
and, for $\bP$-almost every $\omega \in \Omega$,
\[
-I(0)-\varepsilon \leq 
\liminf_{N\to\infty}\frac{1}{N}\log P_{0,\omega}(A^\varepsilon_{N,\n_\varepsilon,\p_\varepsilon})\leq \limsup_{n \to \infty} \frac{1}{N}\log P_{0,\omega}(X_N=0) \leq -I(0).
\]
\end{itemize}
\end{theorem}

\begin{remark} Part (i) of Theorem~\ref{theorem1} was originally proven for dimension $d=1$ by Comets, Gantert and Zeitouni in \cite{CGZ00} and for any dimension by
  Varadhan in~\cite{V03}. We include it in the statement of the theorem for completeness.
  
  \end{remark}

  \begin{remark} The smallness of $\delta_0$ in part $(iii)$ of Theorem \ref{theorem1} depends on the dimension $d$ and on $\varepsilon$ through  the probability per unit volume of realizing the periodic environment $\omega_{\bf n_\varepsilon, \bf p_\varepsilon}$.
  \end{remark}
  
\begin{remark} A statement analogous to part (iii) in the case of nestling random walks was proven by Zerner in \cite{Z98}. In that case $I(0)=0$ and essentially a good strategy for the random walk to realize the event $\{X_n=0\}$ is to move in time $O(N/\log N)$ from $0$ to a ball of radius $O((\log N)^{1/d})$ within a distance $N/\log N$ where the environment is constructed to produce a trap (a {\it nest} in the language of \cite{Z98}), stay there a time of order $O(N)$, and come back to $0$ again in a time of order $O(N/\log N)$.
\end{remark}

\begin{remark} In part $(iii)$ of Theorem \ref{theorem1}, the
periodic environment can be chosen as alternating parallel strips perpendicular to the velocity of the RWRE, where each strip is composed of a homogeneous environment so that the jump probabilities are all equal to some given element of the support of $\omega(0)$ (see Figures \ref{figure1sb} and \ref{figure-tiling}).
\end{remark}

\section{Large deviations for RWPE}

In this section, we derive an explicit formula for the large deviation rate
function at the origin of random walks in (space-)periodic environments. As a first step, we shall study the analogous problem for random walks in time-periodic environments.

\subsection{Large deviations for random walks in time-periodic 
  environments} \label{sec:LDRWTP}
  
 Throughout this section we will work with random walks on $\bZ^d$ evolving according to a time-periodic jump rule. More precisely, let us fix $k \in \bN$ and a $k$-tuple $\q:=(\vec{q}_0,\dots,\vec{q}_{k-1}) \in (\cP_\bV)^k$ of (not necessarily different) probability vectors and define the sequence of environments $\omega=(\omega_n)_{n \in \bN_0} \subseteq \Omega$ by the formula
\[
\omega_n(x,e):=q_{n \,\text{mod}\, k}(e) \qquad \text{ for each $x \in \bZ^d$, $e \in \bV$}.
\] 
We then define the {\it random walk in the time-periodic environment $\omega$} (RWTPE)
starting from $x \in \bZ^d$ as the (time inhomogeneous) Markov chain $(X_n)_{n \in \bN_0}$ with law $Q_{x,\q}$ given by $Q_{x,\q}(X_0=x)=1$ and, for each $n \in \bN_0$, $y \in \bZ^d$ and $e \in \bV$, 
\[
Q_{x,\q}(X_{n+1}=y+e|X_n=y)=\omega_n(y,e).
\] For convenience, we shall sometimes call $(X_n)_{n \in \bN_0}$ a $\q$-RWTPE.  We define the {\it support} of the 
$\q$-RWTPE as the set $\cS$ consisting of the different elements in the finite sequence $\vec{q}_0,\dots,\vec{q}_{k-1}$, and we enumerate its different~elements by $\sigma_1,\dots,\sigma_{j}$ so that we can write
\begin{equation} \label{eq:SDef}
\cS=\{\sigma_1,\ldots,\sigma_{j}\},
\end{equation} where we note that   $\sigma_i \neq \sigma_{i'}$ for all $i\neq i'$. For each $i=1,\dots,j$, we define the {\it multiplicity} of $\sigma_i$ as the number of times $k_i$
that $\sigma_i$ appears in the $k$-tuple $\q$ and the \textit{frequency} of $\sigma_i$ as $\lambda_{i}:=\frac{k_i}{k}$. Note that
\[
k_1+\cdots+k_{j}=k
\] 
and $\overline{\lambda}:=(\lambda_1,\dots,\lambda_{j})\in\Delta_j \cap \bQ^j$, where $\Delta_j$ denotes the $j$-simplex 
\begin{equation}
\label{simplexj}
\Delta_{j}:=\Big\{(t_1,\dots,t_{j}) \in [0,\infty)^{j} :\sum_{i=1}^{j} t_i=1\Big\}.
\end{equation}
Given a vector
$\overline{\lambda}=(\lambda_1,\ldots,\lambda_{j}) \in \Delta_{j} \in \bQ^{j}$ as above,
we will say that the RWTPE has \textit{frequency}~$\overline{\lambda}.$ 
Observe that when $\cS=\{\sigma\}$ is a singleton, the resulting RWTPE is merely a simple random walk on $\bZ^d$ with jump distribution $\sigma$, or $\sigma$-RW, for short. Recall that, if we consider the logarithmic moment generating function $\Lambda_\sigma : \bR^d \to \bR$ given by the formula
\begin{equation}\label{eq:lgf1}
\Lambda_\sigma(\theta):= \log \Big(\sum_{e \in \bV} \mathrm{e}^{\langle\theta,e\rangle}\sigma(e)\Big) 
\end{equation} for any $\theta \in \bR^d$, then, by Cram\'er's theorem, the $\sigma$-RW satisfies a large deviation principle with rate function~$I_\sigma$ given for each $x \in \bR^d$ by the formula
\[
I_\sigma(x):=\sup_{\theta \in \bR^d} \big(\langle \theta, x\rangle - \Lambda_\sigma(\theta)\big).
\] In particular, it is straightforward to verify that, for any \textit{non-degenerate} $\sigma \in \cP_{\bV}$, i.e., such that $\sigma(e) > 0$ for all $e \in \bV$, one has that \begin{equation}\label{eq:rateat0}
I_\sigma(0):= -\inf_{\theta \in \bR^d}\Lambda_\sigma(\theta)=-\log \Big( \sum_{e \in \bV} \sqrt{ \sigma(e)\sigma(-e)}\Big),
\end{equation} since in this case $\Lambda_{\sigma}$ is finite, continuous and satisfies $\lim_{|\theta| \to \infty} \Lambda_{\sigma}(\theta)=+\infty$, so that it must attain its global minimum at its unique critical point (which can be immediately computed and thus yields \eqref{eq:rateat0}). Furthermore, if $(X_n)_{n \in \bN_0}$ is any $\q$-RWTPE with support $\cS=\{\sigma_1,\dots,\sigma_j\}$ and frequency $\overline{\lambda}=(\lambda_1,\dots,\lambda_j)$, then a simple computation gives that the limit
\[
\Lambda_\q (\theta):=\lim_{n\to\infty}\frac{1}{n}\log 
Q_{0,\q}\big[\rme^{\langle \theta, X_n\rangle}\big]=\sum_{i=1}^j \lambda_i \Lambda_{\sigma_i}(\theta)
\] exists and is finite for every $\theta \in \bR^d$, where $Q_{0,\q}[X]$ denotes the expectation of $X$ with respect to the probability measure $Q_{0,\q}$. Furthermore,  $\Lambda_\q$ is everywhere differentiable as a function  of $\theta$. It then follows from the G\"artner-Ellis theorem \cite[Exercise~2.3.20]{DZ98}, that the $\q$-RWTPE satisfies a large deviation principle with rate function $I_\q$ given by
\[
I_\q(x):=\sup_{\theta \in \bR^d} \big(\langle \theta, x\rangle - \Lambda_\q(\theta)\big).
\] 
Recalling the notation $\cP_\bV^{(\kappa)}$ from \eqref{eq:Pkappa}, we are now ready to state the main result of this section.

\begin{proposition}
\label{prop31}
Fix $\kappa > 0$ and a nonempty set $\mathrm{A} \subseteq \cP_\bV^{(\kappa)}$ of probability vectors. Then, for each $\varepsilon > 0$ there exists some period $\q_\varepsilon$ such that the corresponding $\q_\varepsilon$-RWTPE satisfies 
\[
\left|I_{\q_\varepsilon}(0)- \inf_{\sigma\in 
    {\mathcal K}_{\mathrm{A}}}I_{\sigma}(0)\right|\le\varepsilon. 
\]
\end{proposition}

For the proof of Proposition~\ref{prop31} we will need the following lemma (recall $\Delta_j$ is the $j$-dimensional simplex, cf.\ \eqref{simplexj}).

\begin{lemma}\label{lem:saddle} Fix $\kappa > 0$ and let $\sigma_1,\dots,\sigma_j \in \cP_\bV^{(\kappa)}$ be all distinct. Define the mapping  $\Lambda : \Delta_j \times \bR^d \to \R$ by the formula
\[
\Lambda(t,\theta) := \Lambda_{\langle t , \sigma\rangle}(\theta)=\log \Big(  \sum_{i=1}^j t_i \sum_{e \in \bV} \rme^{\langle \theta,e\rangle} \sigma_i(e)\Big),
\] where $\langle t , \sigma \rangle:=\sum_{i=1}^j t_i \sigma_i \in \cP_{\bV}$ for $t \in \Delta_j$. Then, $\Lambda$ has at least one saddle point, by which we refer to the (global) property that there exists some $(t^*,\theta^*) \in \Delta_j \times \bR^d$ such that $\Lambda(t,\theta^*) \leq \Lambda(t^*,\theta^*) \leq \Lambda(t^*,\theta)$ for all $t \in \Delta_j$, $\theta \in \bR^d$. Observe that such $(t^*,\theta^*)$ necessarily
satisfies
\[
\Lambda(t^*,\theta^*) = \max_{t \in \Delta_j} \inf_{\theta \in \bR^d} \Lambda(t,\theta) = \inf_{\theta \in \bR^d} \max_{t \in \Delta_j} \Lambda(t,\theta).
\] 
\end{lemma}

\begin{proof} By  direct computation, it is straightforward to verify that:
\begin{enumerate}
\item [$\bullet$] $(t,\theta) \mapsto \Lambda(t,\theta)$ is continuous on $\Delta_j \times \bR^d$,
\item [$\bullet$] $t \mapsto \Lambda(t,\theta)$ is concave for all $\theta \in \bR^d$, 
\item [$\bullet$] $\theta \mapsto \Lambda(t,\theta)$ is convex for all $t \in \Delta_j$.
\end{enumerate} Therefore, by Sion's minimax theorem \cite{Si58}, it follows that the minimax equality holds, i.e.,
\[
\max_{t \in \Delta_j} \inf_{\theta \in \bR^d} \Lambda(t,\theta) = \inf_{\theta \in \bR^d} \max_{t \in \Delta_j} \Lambda(t,\theta).
\] Now let us define the function $R_1 : \bR^d \to \bR$ via
\[
R_1(\theta):= \max_{t \in \Delta_j} \Lambda(t, \theta)
\] 
and observe that
\[
R_1(\theta)= \log \Big( \max_{i=1,\dots,j} \sum_{e \in \bV} \rme^{\langle \theta,e\rangle} \sigma_i(e)\Big) = \max_{i=1,\dots,j} \Lambda_{\sigma_i}(\theta).
\] Hence, since it is straightforward to check from \eqref{eq:lgf1} that each $\Lambda_{\sigma_i}$ is finite, continuous and satisfies $\lim_{|\theta| \to \infty} \Lambda_{\sigma_i}(\theta)=+\infty$, it follows that $R_1$  satisfies these three properties as well and, as such, the set
\[
X_*:=\{ \theta \in \bR^d : R_1(\theta) = \inf R_1 \}
\] of minimizers is non-empty. 

On the other hand, if we consider the function $R_2:\Delta_j \to \R$ given by
\[
R_2(t):= \inf_{\theta \in \bR^d} \Lambda(t,\theta),
\] then, since $\Lambda(t,\theta)=\Lambda_{\langle t,\sigma\rangle}(\theta)$, by \eqref{eq:rateat0} we see that
\[
R_2(t)= \log \left( \sum_{e \in \bV} \sqrt{ \sum_{i,k=1}^j t_i t_k \sigma_i(e)\sigma_k(-e)}\right).
\] In particular, $R_2$ is finite and continuous, so that, since $\Delta_j$ is compact, we obtain that
\[
Y_*:=\{ t \in \Delta_j : R_2(t) = \sup R_2\} 
\] is non-empty. It now follows from \cite[Proposition 3.4.1]{B09} that any point in $X_* \times Y_*$ is a saddle point. 
\end{proof}

We are now ready to prove Proposition~\ref{prop31}.

\begin{proof}[Proof of Proposition~\ref{prop31}] We start by noting that given any $\varepsilon > 0$ we may choose $j \in \N$ and distinct $\sigma_1,\dots,\sigma_j \in \mathrm{A}$ such that, by definition of convex hull, 
\begin{equation} \label{eq:KAS}
\Big|\inf_{\sigma \in \mathcal{K}_{\mathrm{A}}} I_\sigma(0) - \inf_{\sigma \in \mathcal{K}_{\cS}} I_\sigma(0)\Big| \leq \frac{\varepsilon}{2},
\end{equation}
where we define $\cS:=\{\sigma_1,\ldots,\sigma_j\}$.

Next, let $\Lambda: \Delta_j \times \bR^d \to \R$ be the mapping from  Lemma~\ref{lem:saddle} (for this particular choice of $\sigma_1,\dots,\sigma_j$). Then, by this very same lemma, there exists $(t^*,\theta^*)$ such that
\begin{equation}\label{eq:saddle2}
\Lambda(t^*,\theta^*)=\inf_{\theta \in \bR^d} \max_{t \in \Delta_j} \Lambda(t,\theta)=\max_{t \in \Delta_j}\inf_{\theta \in \bR^d}  \Lambda(t,\theta).
\end{equation} In particular, upon noticing that $-I_{\langle t,\vec{\sigma}\rangle}(0)=\inf_{\theta \in \bR^d} \Lambda(t,\theta)$, we see that
\[\Lambda(t^*,\theta^*)=-\min_{t \in \Delta_j} I_{\langle t , \vec{\sigma}\rangle} (0) = - \inf_{\sigma \in \mathcal{K}_{\cS}}I_\sigma(0). 
\]
Now, if we write $t^*=(t_1^*,\dots,t_j^*)$ and define $\Lambda^*_{\vec{\sigma}}:\bR^d \to \R$ by the formula
\begin{equation}\label{eq:expcomb}
\Lambda^*_{\vec{\sigma}}(\theta):=\sum_{i=1}^j t_i^* \Lambda_{\sigma_i}(\theta),
\end{equation} then our next step will be to show that $\Lambda^*_{\vec{\sigma}}$ has a global minimum at $\theta^*$ which satisfies $\Lambda^*_{\vec{\sigma}}(\theta^*)=\Lambda(t^*,\theta^*)$, so that
\begin{equation}\label{eq:equal}
\min_{\theta \in \bR^d} \Lambda^*_{\vec{\sigma}}(\theta) = - \inf_{\sigma \in \mathcal{K}_{\cS}} I_\sigma(0).
\end{equation} To this end, we first observe that, $\Lambda_{\sigma_i}(\theta^*)=\max_{k=1,\dots,j}\Lambda_{\sigma_{k}}(\theta^*)$ for all $i$ with $t^*_i> 0$. Indeed, since for all $t \in \Delta_j$ we have
\begin{equation}\label{eq:repexp}
\rme^{\Lambda(t,\theta^*)}=\sum_{i=1}^j t_i \rme^{\Lambda_{\sigma_i}(\theta^*)},
\end{equation} if there existed $i,k$ such that $t^*_i > 0$ and $\Lambda_{\sigma_i}(\theta^*) < \Lambda_{\sigma_k}(\theta^*)$ then this would imply that
\[
\Lambda( t^* + t^*_ie_k - t^*_ie_i, \theta^*) > \Lambda(t^*,\theta^*),
\] which would contradict the fact that $(t^*,\theta^*)$ is a saddle point. 
Thus, it follows from this observation and \eqref{eq:repexp} that, if we write $\overline{\Lambda}(\theta^*):=\max_{k=1,\dots,j} \Lambda_{\sigma_k}(\theta^*)$, then 
\[
\frac{\partial \Lambda^*_{\vec{\sigma}}(\theta)}{\partial \theta_k}\bigg|_{\theta=\theta^*} = \sum_{i=1}^j t_i^* \rme^{-\overline{\Lambda}(\theta^*)} \frac{\partial \rme^{\Lambda_{\sigma_i}(\theta)}}{\partial \theta_k}\bigg|_{\theta=\theta^*}=\rme^{-\overline{\Lambda}(\theta^*)}\frac{ \partial \rme^{\Lambda_{\langle t^*,\vec{\sigma}\rangle}(\theta^*)}}{\partial \theta_k}=0
\] since $(t^*,\theta^*)$ is a saddle point of $\Lambda$ and thus $\theta^*$ minimizes $\Lambda_{\langle t^* ,\vec{\sigma}\rangle}$. Arguing as with $R_1$ in the proof of Lemma~\ref{lem:saddle}, we can show that $\Lambda^*_{\vec{\sigma}}$ is finite, continuous and also satisfies $\lim_{|\theta| \to \infty} \Lambda^*_{\vec{\sigma}}(\theta)=\infty$, which implies that its critical point $\theta^*$ must be a global minimum. In addition, by this same observation and \eqref{eq:expcomb}, we obtain that $\Lambda^*_{\vec{\sigma}}(\theta^*)=\overline{\Lambda}(\theta^*)$. Finally, since we have that
\[
\Lambda(t^*,\theta^*)=\max_{t \in \Delta_j} \Lambda(t,\theta^*) = \overline{\Lambda}(\theta^*)
\] as in the proof of Lemma~\ref{lem:saddle}, we conclude that $\Lambda^*_{\vec{\sigma}}(\theta^*)=\Lambda(t^*,\theta^*)$ and hence \eqref{eq:equal} holds. 

To conclude the proof we observe that, since each $\Lambda_{\sigma_i}$ is finite, continuous and verifies both $\Lambda_{\sigma_i}(0)=0$ and $\lim_{|\theta| \to \infty} \Lambda_{\sigma_i}(\theta)=\infty$, it is straightforward to verify that, as $n \to \infty$, for any $t^{(n)}=(t_1^{(n)},\dots,t_j^{(n)}) \in \Delta_j$ with $t^{(n)} \rightarrow t^*$ we have
\[
\inf_{\theta \in \bR^d} \sum_{i=1}^j t_i^{(n)} \Lambda_{\sigma_i}(\theta) \longrightarrow
\inf_{\theta \in \bR^d} \sum_{i=1}^j t^*_i \Lambda_{\sigma_i}(\theta) = \Lambda^*_{\vec{\sigma}}(\theta^*)=-\inf_{\sigma \in \mathcal{K}_{\cS}}I_\sigma(0).
\] 
In particular, we see that for each sequence $(t^{(n)})_{n \in \N} \subseteq \Delta_j \cap \bQ^j$ with $t^{(n)} \rightarrow t^*$ as $n \to \infty$, we have that for any $\q_n$-RWTPE with support $\cS$ and frequency $t^{(n)}$,
\[
-I_{\q_n}(0)=\inf_{\theta \in \bR^d} \Lambda_{\q_n}(\theta)= \inf_{\theta \in \bR^d} \sum_{i=1}^j t^{(n)}_i\Lambda_{\sigma_i}(\theta) \longrightarrow -\inf_{\sigma \in \mathcal{K}_{\cS}}I_\sigma(0).
\] 
In combination with \eqref{eq:KAS}, the result now immediately follows.
\end{proof}

\subsection{Large deviations for random walks in space-periodic 
  environments} \label{sec:LDRWSP}

Consider a random walk $(X_n)_{n \in \bN_0}$ in the space-periodic environment induced by a period $(\n,\p)$, which we assume to be uniformly elliptic. Associated to this walk we can consider the Markov chain $(Y_n)_{n \in \bN_0}$
defined as
\[
Y_n:=X_n \,\text{mod}\, {\bf n}
\]
with finite state space $\mathbf{B}_{\bf n}$. The uniform ellipticity assumption entails that $(Y_n)_{n \in \bN_0}$ is irreducible. By choosing ${\bf n}$ with the correct parity, we can  also assume it is aperiodic. It therefore has a unique 
invariant measure on $\mathbf{B}_{\bf n}$ which we will denote by $\mu_{\n,\p}$.

In the sequel, we will need to construct a specific type of space-periodic environment, generated by different
environments on parallel strips. To this end, given  
$
u\in\mathbb Q^d \setminus \{0\}$ and $r_1,r_2 \in \bR$ with $r_1 < r_2$, we define the $(r_1,r_2)$-strip perpendicular to $u$ as
\[
S_{u,r_1,r_2}:=\{x\in\mathbb Z^d: r_1\le  \langle x, u \rangle <r_2\}.
\]
Next, we fix $r_0=0<r_1<\cdots<r_j$ and define a partition of $\mathbb Z^d$ given by the strips~$(S_i)_{i\in\mathbb Z}$, 
where $S_i:=S_{u,nr_j + r_{l-1},nr_j+r_l}$ 
if $i=nj+l$ with $l\in\{1,\ldots,j\}$. 
We call the partition $(S_i)_{i\in\mathbb Z}$ a {\it $(u;r_1,\dots,r_j)$-strip periodic tiling}. 

This tiling itself can be  decomposed into finite blocks or {\em tiles}, which are translated copies of a particular hyper-rectangle $R$ which we introduce next. We define recursively the set of vertices of one face
of $R$ as follows: let $x_0=0$ and, for $0\le n\le d-1$, we choose its $(n+1)$-th vertex as any point
\[
x_{n+1}\in{\rm argmin} \left\{|x|_2: x\in\mathbb Z^d: \langle u, x \rangle=0, \langle x_i, x\rangle =0, x\ne x_i \ {\rm for}\ 0\le i\le n\right\},
\]
and observe that the argmin above exists precisely because $u$ has rational coordinates.
We then define the \textit{base} $F$ of the hyper-rectangle $R$ as 
\begin{equation}
\label{basef}
F:=\Big\{x\in\mathbb R^d: \sum_{n=0}^{d-1} s_n x_n, 0\le s_n< 1, \sum_{n=0}^{d-1}s_n=1\Big\}.
\end{equation}
We now define the hyper-rectangle $R$ from its base $F$ via the formula
\begin{equation}
\label{bst}
R:=\big\{y\in\mathbb Z^d: y=x+\alpha u\,,\, x\in F\,,\, 0\le \alpha < r_j\big\}.
\end{equation}
We call $R$ the {\it basic tile} of the $(u;r_1,\dots,r_j)$-strip-periodic tiling $(S_i)_{i \in \bZ}$
(see Figure \ref{figure-tiling}). 

Finally, let us fix $u \in \bQ^d \setminus \{0\}$ and consider distinct vectors $\sigma_1,\ldots,\sigma_j\in\mathcal P_{\mathbb V}^{(\kappa)}$. Given a $(u;r_1,\dots,r_j)$-strip-periodic tiling, we define a (space-)periodic environment
on $\mathbb Z^d$ via
\[
\omega^*(x)=\sigma_{i \,\text{mod}\, j},
\]
whenever $x \in S_i$. We call this a {\it $(u;r_1,\dots,r_j)$-strip periodic environment} generated by $\sigma_1,\ldots,\sigma_j$. Observe that we can write $\omega^*=\omega_{\n,\p}$ for some specific period $(\n,\p)$ depending only on the $j$-tuple $(r_1,\dots,r_j)$. In particular, $\omega^*$ defines
a particular type of RWPE, which we shall denote by $(X_n)_{n \in \bN_0}$ (see Figure~\ref{figure-tiling}). Also, if we project $(X_n)_{n \in \N_0}$ onto the hyper-rectangle $R$ via periodicity, we obtain a random walk $(Z_n)_{n \in \N}$ on the finite state space $R$, i.e., $Z_n$
is given by 
\begin{equation}
\label{zne}
Z_n:=\hat X_n,
\end{equation}
where $\hat x$ denotes the equivalence class defined by the basic tile $R$ on $\mathbb Z^d$.

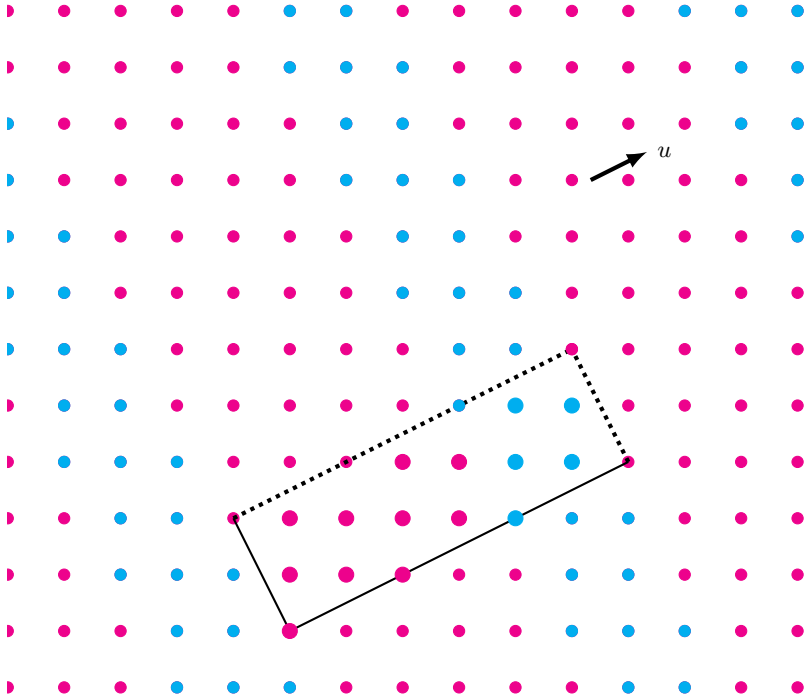
\begin{figure}[ht]
  \centering 
  \begin{tikzpicture}
    \coordinate (Origin)   at (0,0); 
    \coordinate (XAxisMin) at (-3,0); 
    \coordinate (XAxisMax) at (5,0); 
    \coordinate (YAxisMin) at (0,-2); 
    \coordinate (YAxisMax) at (0,5);

    \clip (-3,-2) rectangle (10cm,10cm);

    \coordinate (Bone) at (0,2); 
    \coordinate (Btwo) at (2,-2); 

    \foreach \x in {-5,-4.5,-4,...,5}{
      \foreach \y in {-5,-4.5,-4,...,5}{
        \node[draw,circle,magenta, inner sep=1.5pt,fill] at (1.5*\x,1.5*\y) {}; 
      }
    }

    \draw [thick,black] (0.75,-0.75) 
    -- (5.25,1.5) node [below left] {};
    \draw [ ultra thick,dotted, black] (5.25,1.5) 
   -- (4.5,3) node [below left] {};
             \draw [ thick,black] (-0,0.75) 
  -- (0.75,-0.75) node [below left] {};

       \draw [ultra thick,dotted, black] (0,0.75) 
        -- (4.5,3) node [below left] {};

   \node[draw,circle,magenta, inner sep=2pt,fill] at (0.75,-0.75) {};
   \node[draw,circle,magenta, inner sep=2pt,fill] at (2.25,0) {}; 
  \node[draw,circle,magenta, inner sep=2pt,fill] at (1.5,0) {}; 
 \node[draw,circle,magenta, inner sep=2pt,fill] at (3,0.75) {}; 
 \node[draw,circle,magenta, inner sep=2pt,fill] at (0.75,-0) {};
 \node[draw,circle,magenta, inner sep=2pt,fill] at (0.75,0.75) {};
 \node[draw,circle,magenta, inner sep=2pt,fill] at (1.5,0.75) {};
 \node[draw,circle,magenta, inner sep=2pt,fill] at (2.25,0.75) {};
 \node[draw,circle,magenta, inner sep=2pt,fill] at (2.25,1.5) {};
 \node[draw,circle,magenta, inner sep=2pt,fill] at (3,1.5) {};

\draw [ultra thick,-latex,black] (4.75,5.25) 
   -- (5.5,5.625) node [right] {$u$}; 
   \node[draw,circle,cyan, inner sep=2pt,fill] at (3.75,0.75) {}; 
 \node[draw,circle,cyan, inner sep=2pt,fill] at (4.5,1.5) {}; 
 \node[draw,circle,cyan, inner sep=2pt,fill] at (3.75,1.5) {};
 \node[draw,circle,cyan, inner sep=2pt,fill] at (3.75,2.25) {};
 \node[draw,circle,cyan, inner sep=2pt,fill] at (4.5,2.25) {};

 \node[draw,circle,magenta, inner sep=1.5pt,fill] at (4.5,3) {};
   \node[draw,circle,cyan, inner sep=1.5pt,fill] at (3,2.25) {}; 
   \node[draw,circle,cyan, inner sep=1.5pt,fill] at (3,3) {};
   \node[draw,circle,cyan, inner sep=1.5pt,fill] at (3,3.75) {};
        \node[draw,circle,cyan, inner sep=1.5pt,fill] at (3.75,3.75)
        {};
            \node[draw,circle,cyan, inner sep=1.5pt,fill] at (3.75,3) {}; 
            \node[draw,circle,cyan, inner sep=1.5pt,fill] at (2.25,3.75) {};
              \node[draw,circle,cyan, inner sep=1.5pt,fill] at
              (2.25,4.5) {};
                            \node[draw,circle,cyan, inner
                            sep=1.5pt,fill] at (2.25,5.25) {};

               \node[draw,circle,cyan, inner sep=1.5pt,fill] at 
               (3,5.25) {};
               \node[draw,circle,cyan, inner sep=1.5pt,fill] at (3,4.5) {};
  \node[draw,circle,cyan, inner sep=1.5pt,fill] at 
  (1.5,5.25) {};

  \node[draw,circle,cyan, inner sep=1.5pt,fill] at 
  (1.5,6) {};
    \node[draw,circle,cyan, inner sep=1.5pt,fill] at 
    (1.5,6.75) {};
     \node[draw,circle,cyan, inner sep=1.5pt,fill] at 
     (2.25,6.75) {};
          \node[draw,circle,cyan, inner sep=1.5pt,fill] at 
               (2.25,6) {};
  \node[draw,circle,cyan, inner sep=1.5pt,fill] at 
  (0.75,6.75) {};

   \node[draw,circle,cyan, inner sep=1.5pt,fill] at 
   (0.75,7.5) {};
   
   \node[draw,circle,cyan, inner sep=1.5pt,fill] at 
   (1.5,7.5) {};

     \node[draw,circle,cyan, inner sep=1.5pt,fill] at 
     (4.5,0.75) {};
     
     \node[draw,circle,cyan, inner sep=1.5pt,fill] at 
     (4.5,0) {};

     \node[draw,circle,cyan, inner sep=1.5pt,fill] at 
     (4.5,-0.75) {};

     \node[draw,circle,cyan, inner sep=1.5pt,fill] at 
     (5.25,-0.75) {};

     \node[draw,circle,cyan, inner sep=1.5pt,fill] at 
     (5.25,-1.5) {};

     \node[draw,circle,cyan, inner sep=1.5pt,fill] at 
   (5.25,0) {};

     \node[draw,circle,cyan, inner sep=1.5pt,fill] at 
   (5.25,0.75) {};

     \node[draw,circle,cyan, inner sep=1.5pt,fill] at 
     (6,-1.5) {};

     \node[draw,circle,cyan, inner sep=1.5pt,fill] at 
   (6,-0.75) {};

      \node[draw,circle,cyan, inner sep=1.5pt,fill] at 
      (6,7.5) {};

         \node[draw,circle,cyan, inner sep=1.5pt,fill] at 
         (6.75,7.5) {};
                  \node[draw,circle,cyan, inner sep=1.5pt,fill] at 
                  (7.5,7.5) {};
                               \node[draw,circle,cyan, inner sep=1.5pt,fill] at 
                               (7.5,6.75) {};
                                                      \node[draw,circle,cyan, inner sep=1.5pt,fill] at 
   (6.75,6.75) {}; 
                                      \node[draw,circle,cyan, inner sep=1.5pt,fill] at 
                                      (6.75,6) {};
                                                                       \node[draw,circle,cyan, inner sep=1.5pt,fill] at 
                                                                       (7.5,6) {};

                                        \node[draw,circle,cyan, inner sep=1.5pt,fill] at 
                                        (7.5,5.25) {};
                                          \node[draw,circle,cyan, inner sep=1.5pt,fill] at 
   (7.5,4.5) {};

  \node[draw,circle,cyan, inner sep=1.5pt,fill] at 
  (0.75,-1.5) {};

  \node[draw,circle,cyan, inner sep=1.5pt,fill] at 
  (0,-1.5) {};

  \node[draw,circle,cyan, inner sep=1.5pt,fill] at 
  (-0.75,-1.5) {};
  
  \node[draw,circle,cyan, inner sep=1.5pt,fill] at 
  (-0,-0.75) {}; 
  
  \node[draw,circle,cyan, inner sep=1.5pt,fill] at 
  (-0,0) {};

  \node[draw,circle,cyan, inner sep=1.5pt,fill] at 
  (-0.75,-0.75) {}; 
  
  \node[draw,circle,cyan, inner sep=1.5pt,fill] at 
  (-0.75,0) {};

  \node[draw,circle,cyan, inner sep=1.5pt,fill] at 
  (-0.75,0.75) {};
  
  \node[draw,circle,cyan, inner sep=1.5pt,fill] at 
  (-0.75,1.5) {};

  \node[draw,circle,cyan, inner sep=1.5pt,fill] at 
  (-1.5,1.5) {}; 
  \node[draw,circle,cyan, inner sep=1.5pt,fill] at 
  (-1.5,0.75) {}; 
  \node[draw,circle,cyan, inner sep=1.5pt,fill] at 
  (-1.5,0) {}; 
  \node[draw,circle,cyan, inner sep=1.5pt,fill] at 
  (-1.5,2.25) {}; 
    \node[draw,circle,cyan, inner sep=1.5pt,fill] at 
    (-1.5,3) {};

      \node[draw,circle,cyan, inner sep=1.5pt,fill] at 
  (-2.25,3) {}; 
  \node[draw,circle,cyan, inner sep=1.5pt,fill] at 
  (-2.25,2.25) {}; 
  \node[draw,circle,cyan, inner sep=1.5pt,fill] at 
  (-2.25,1.5) {}; 
  \node[draw,circle,cyan, inner sep=1.5pt,fill] at 
  (-2.25,4.5) {}; 
  \node[draw,circle,cyan, inner sep=1.5pt,fill] at 
  (-2.25,3.75) {};

      \node[draw,circle,cyan, inner sep=1.5pt,fill] at 
  (-3,3) {}; 
  \node[draw,circle,cyan, inner sep=1.5pt,fill] at 
  (-3,5.25) {}; 
  \node[draw,circle,cyan, inner sep=1.5pt,fill] at 
  (-3,6) {}; 
  \node[draw,circle, cyan , inner sep=1.5pt,fill] at 
  (-3,4.5) {}; 
  \node[draw,circle,cyan, inner sep=1.5pt,fill] at 
  (-3,3.75) {};

  \end{tikzpicture}
  \caption{Strip periodic tiling in $\mathbb Z^2$ composed of $2$ strips perpendicular to a vector $u$ with rational coordinates. The basic tile $A$ (see \eqref{bst}) is
    composed by the bold magenta and blue lattice points inside the black rectangle. Its base $F$ (see \eqref{basef}) is the lower-left side of the rectangle.}
  \label{figure-tiling}
\end{figure}

\begin{lemma} \label{l32} Fix $\kappa > 0$, $j\ge 1$ and $\cS:=\{\sigma_1,\dots,\sigma_j\} \subseteq \cP_{\bV}^{(\kappa)}$ with all $\sigma_i$'s distinct. Assume that there exists some $u \in \bQ^d$, with $|u|_2=1$,  such that
\begin{equation}
\label{vsu}
\langle \rmd(\sigma_i), u \rangle =0
\end{equation}
for all $1\le i\le j$, where $\rmd(\sigma_i)$ is the drift of $\sigma_i$ as defined in \eqref{eq:defdrift}. Then, given $\varepsilon > 0$, for any $\overline{\lambda}=(\lambda_1,\dots,\lambda_j) \in \Delta_j$  
there exists some periodic environment $\omega_{\n,\p}$ with $\p=(\vec{p}(x)) \in \cS^{\mathbf{B}_\n}$ such that $P_{0,\omega_{\n,\p}}$-almost surely, for every $1 \leq m \leq j$,
\begin{equation}
\label{emp}
\limsup_{n\to\infty}\left|\frac{\sum_{i=0}^n1_{\{X_i\in S_m\}}}{n}-\lambda_m\right|\le\varepsilon,
\end{equation}
where $S_m$ is the set of sites $x\in \bZ^d$ such that $\omega_{\n,\p}(x)=\sigma_m$. As a matter of fact, $\omega_{\n,\p}$ can be chosen as a $(u;r_1,\dots,r_j)$-strip periodic environment generated by the $\sigma_1,\ldots,\sigma_j$ for some $r_1,\dots,r_j$.
\end{lemma}

\begin{proof} Consider some fixed $(u;r_1,\dots,r_j)$-strip periodic environment generated by $\sigma_1,\ldots,\sigma_j$. Observe that the strip $S_i$ in the corresponding $(u;r_1,\dots,r_j)$-strip periodic tiling has width $r_l-r_{l-1}$ if $i=nj+l$ for some $l \in \{1,\dots,j\}$. We will assume throughout that $r_1,\ldots, r_j\in\mathbb N$. This periodic environment defines a RWPE, which we denote by $(X_n)_{n \in \bN_0}$. We claim that this is the desired RWPE satisfying \eqref{emp}. We will carry out the proof of this in a number of steps, which we outline below.\\
 
 \noindent {\it Step 1.} First, we will define a one-dimensional random walk given by the projection of the RWPE $(X_n)_{n \in \N_0}$ onto the line generated by $u$. Let $R$ be the basic tile in the chosen $(u;r_1,\dots,r_j)$-strip periodic tiling (see Figure \ref{figure-tiling}) and let $(Z_n)_{n \in \bN_0}$ be the random~walk induced by $(X_n)_{n \in \bN_0}$
 on $R$, see \eqref{zne}. Now define $W=(W_n)_{n \in \bN_0}$ by the formula
 \[
W_n:=\langle Z_n, u \rangle.
\] Note that $W$ takes values in the set $\mathbb W:=\{\langle x, u\rangle : x\in R\}$ and that, if we think of $\mathbb{W}$ as having periodic boundary conditions, then $W$ has independent increments, each supported on the set $\{ \langle e , u \rangle : e \in \bV\}$ and thus bounded from above in absolute value by some constant $\ell_u$ depending only on $u$. 
 Observe that, by the fact that the drifts $\rmd(\sigma_i)$ are all perpendicular~to~$u$,  (cf.\ \eqref{vsu}), $W$
 is a martingale. Notice also that, since $u$ is a unit vector, the number of points per unit length in $\mathbb W$ is some constant integer $L \in \N$ depending only on $u$.
In particular, we have that 
\[
|\mathbb W|=Lr_j.
\] We will show in the next steps that $W$ has an invariant measure which is
 approximately constant on each region of constant jump probabilities, i.e. on each set $R\cap \langle S_i, u\rangle$.\\

\noindent {\it Step 2.} To estimate the invariant measure of $(W_n)_{n \in \bN_0}$, we will have to work with a renormalized random walk corresponding to the exit times of $(W_n)_{n \in \bN_0}$ from intervals of a fixed length $2N$. In order to define such random walk, we  will first divide $\mathbb W$ into overlapping intervals of length $2N$.  Without loss of generality, we can assume that $2N$ divides each of the numbers $r_1,\ldots, r_j$ (which will be chosen accordingly later on) so that
\[
\nu_l:=\frac{r_l-r_{l-1}}{2N},
\]
are all integers. Now, let us set $\nu:=\sum_{l=1}^j \nu_l = \frac{r_j}{2N}$, $\overline{\nu}:=2L\nu$ and define for $i=0,\dots,\overline{\nu}-1$ the intervals
\begin{equation} \label{eq:JiDef}
J_i:=[N(i-1),N(i+1)] \subseteq \R.
\end{equation}
Let $V_0:=0$ and set
\[
T_0:=\min\{n> 0: W_n\notin J_0\}.
\] Observe that, by definition of $W$, the point $W_{T_0}$ must either lie in $J_1$ (if $W$ exits $J_0$ to the right) or in $J_{\overline{\nu}-1}$ (if $W$ exits $J_0$ to the left or, to be more precise, if $\langle X_n, u \rangle $ jumps to the left of $0$). Thus, we may set 
\begin{align*}
V_1:=\begin{cases}1 & \text{ if $W_{T_0} \in J_1$}, \\ \overline{\nu}-1 &\text{ if $W_{T_0} \in J_{\overline{\nu}-1}$},\end{cases}    
\end{align*} 
and then, for $k \in \bN$, define recursively
\[
T_k:=\min\Big\{n> T_{k-1}: W_n\notin J_{V_{k-1}}\Big\}
\]
 together with
\[
V_{k+1}:=\begin{cases}V_k+1 \,\,\text{mod}\,\, \overline{\nu} & \text{ if $W_{T_k} \in J_{V_k+1 \,\text{mod}\, \overline{\nu}}$} \\ V_k-1 \,\,\text{mod}\,\, \overline{\nu} &\text{ if $W_{T_k} \in J_{V_k-1 \,\text{mod}\, \overline{\nu}}$}.\end{cases}
\] In other words, $V=(V_k)_{k \in \bN_0}$ is the process taking values on the set $I_{\overline{\nu}}:=\{0,1,\dots,\overline{\nu}-1\}$ with periodic boundary conditions which starts from $0$ and, at time $k$, jumps either one step to the right (modulo $\overline{\nu}$) if $W$ exits the interval $J_{V_k}$ through its right endpoint or one step to the left otherwise.\\

\noindent {\it Step 3.} Using the fact that $(W_k)_{k \in \bN_0}$ is a
 martingale, by the optional stopping theorem we have that, as $N \to \infty$, 
\begin{equation}\label{eq:asympprob}
P(V_{k+1}=i+1| V_k=i)=\frac{1}{2}+O_N(N^{-1}).
\end{equation}
In particular, by continuity we have that, for any fixed values of the ratios $\nu_1,\dots,\nu_j \in \bN$, $V$ has an invariant measure $\pi$ which satisfies the asymptotics  
\begin{equation}
\label{pix}
\pi(i)=\frac{1}{\overline{\nu}}+o_N(1),
\end{equation} for all $i \in I_{\overline{\nu}}$
where $o_N(1)$ is an error term tending to $0$ as $N \to \infty$.\\

\noindent {\it Step 4.} We next show that the average number of steps that
 $W$ takes to exit the interval $J_i$ starting at most distance $\ell_u$ from $iN$ is of order $C_i N^2$, for some positive constant~$C_i$. To do this,
 consider the martingale 
\[
W_n^2-\sum_{i=1}^n f(W_i),
\]
where
\[
f(x)=f_l:=\sum_{e\in\mathbb V}\sigma_l(e) \langle e, u\rangle^2\qquad {\rm whenever}\ r_{l-1}\le \langle x, u\rangle<r_l.
\]
 Define for each $i\ge 0$,
 \[
 \tau_i:=T_i-T_{i-1},
 \]
 with the convention $T_{-1}=0$. Note that the random variables $(\tau_i)_{i\ge 0}$
 are not identically distributed, since the exit point of the walk $W$ from each interval $J_i$ is random (because the precise length of the increment $W_{n+1}-W_n$ depends on the direction of the jump $Z_{n+1}-Z_n$, which is random). To deal with this, we introduce, for $i \in I_{\overline{\nu}}$, the  intervals
 \[
 \tilde{J}_i:=[(i-1)N - 2\ell_u,(i+1)N+2\ell_u],
 \] where we recall that $\ell_u > 0$ is such that the increments of $W$ are bounded above by $\ell_u$. Next, for each $i \in I_{\overline{\nu}}$ we let $W^{(i)}$ be an independent copy of the $W$ starting from~$x_{i}^{(N)}$, the point in $\mathbb{W}$ which is closest to $iN$, and define
 \[
 \tau^+(i):=\inf\{ n \geq 0 : W^{(i)}_n \notin \tilde{J}_i\}.
 \] Finally, for each $i \in I_{\overline{\nu}}$ let $(\tau^+_k)_{k \in \N}$ be a sequence of i.i.d.\ random variables distributed as $\tau^+(i)$ and notice that we may couple $(\tau^+_k)_{k \in \N}$ with $(\tau_k)_{k \in \bN}$ so that, for every $k \in \bN$, on the event $\{V_k=i\}$ one has that
 \begin{equation}\label{eq:coupling}
 \tau_k \leq \tau^+_{N^k_i}(i), 
 \end{equation} where $N^k_i$ denotes the total number of visits of $V$ to site $i$ until time $k$. Indeed, if $W$ exits the interval $J_i$ by reaching some site $x \in J_{i'}$ with $i' \in \{i-1,i+1\}$ (modulo $\overline{\nu}$), then, since $J_{i'}$ is of length $2N$ and $x$ is at distance at most $2\ell_u$ from $x^{(N)}_{i'}$, the time it takes $W$ to exit this interval $J_{i'}$ (starting from any such $x$) is stochastically dominated by $\tau^+(i')$.

 
Now, notice that there are two types of sites $i \in I_{\overline{\nu}}$:

\begin{itemize}
 \item  {\it Type A:} $i$ satisfies 
 $\tilde{J}_i\subseteq [r_{l-1},r_l)$ for some $l=1,\dots,j$.  
 Here, by optional stopping we have that, for all $N$ sufficiently large (depending only on $u$),
 \begin{equation}
 \label{tya1}
 E[\tau^+(i)]=\frac{1}{f_l}E\Big[(W^{(i)}_{\tau^+(i)}-x_i^{(N)})^2\Big] \le \frac{(N+3\ell_u)^2}{f_l}.
 \end{equation}

\item {\it Type B: } $J_i$ is not contained in any of the intervals $[r_{l-1},r_l)$ for $l=1,\dots,j$. Here, since $\min_{l=1,\dots,j} f_l > 0$, by a reasoning  similar to the one for sites of type $A$, we obtain
\[
E[\tau^+(i)]\le a_i N^2,
\] for some positive constant $a_i$. Notice that there are at most $2j$ sites of type $B$.
\end{itemize}
\smallskip

\noindent {\it Step 4.} To continue, we first need to introduce some notation. Define
 for each $1\le l\le j$,
\[
\mathbb H_{l}:=\left\{i\in I_{\overline{\nu}}: J_i \cap [r_{l-1}, r_l)\neq \emptyset \right\}.
\]
 Now note that, for each 
 $1\le l\le j$ and $m \in \N$, the total number of visits of the walk $W$ up to time $T_m$ to the set 
 \[
 \mathbb W_l:=\{x\in\mathbb W:r_{l-1}\le x <r_l\},
 \]
 can be bounded from above using the coupling \eqref{eq:coupling} by
 \[
\sum_{i\in\mathbb H_{l}} \sum_{k=1}^{N^m_i} \tau^+_{k}(i),
 \]
 where we recall that
 \[
N^m_{i}=\sum_{r=1}^m 1(V_r=i).
 \]
 Observe that by the convergence theorem for Markov chains, there exists a constant $C>0$ such that, for each $i \in I_{\overline{\nu}}$
\[
P\left(\left|\frac{1}{m}N_i^m-\pi(i)\right|>\left(\frac{1}{\overline{\nu}}\right)^2\right)\le Ce^{-Cm}.
\] 
By employing the previous display in combination with the first Borel-Cantelli lemma, it follows that eventually in $m$ we have that
\[
\sum_{i\in\mathbb H_{l}} \sum_{k=1}^{N^m_i} \tau^+_{k}(i)\le
\sum_{i\in\mathbb H_{l}} \sum_{k=1}^{n_i^m} \tau^+_{k}(i),
\]
 where
 \[
n_i^m:=m\left(\pi(i)+\left(\frac{1}{\overline{\nu}}\right)^2\right).
 \]
By the law of large numbers and \eqref{pix}, it follows that $P_{0,\omega_{\n,\p}}$-a.s.\ for each $l=1,\dots,j$ we have
\[
\limsup_{m\to\infty}\frac{1}{m}\sum_{n=0}^{T_m} 1_{\mathbb{W}_l}(W_n)
\le
\sum_{i\in \mathbb H_l}\left( \pi(x)+ \left(\frac{1}{\overline{\nu}}\right)^2\right)E[\tau^+(i)].\]
If $i \in \mathbb{H}_l$ is of type $A$, then the summand corresponding to $i$ on the right-hand side of the last display is bounded above by $\big(\frac{1}{\overline{\nu}} + \big( \frac{1}{\overline{\nu}}\big)^2+o_N(1)\big)\frac{(N+3\ell_u)^2}{f_l}$. On the other hand, if $i$ is of type $B$ then its corresponding summand is bounded above by $(\frac{2}{\overline{\nu}}+o_N(1))a_iN^2$. Since there are at most $2j$
sites of type $B$ and at most $2L\nu_l + 4$ sites in $\mathbb{H}_l$ of type $A$, a straightforward computation using that $\overline{\nu} \geq 1$ allows us to conclude that
\[
\limsup_{m\to\infty}\frac{1}{m}\sum_{n=0}^{T_m} 1_{\mathbb{W}_l}(W_n) \leq \left(\frac{\nu_l}{\nu} + \frac{10}{\overline{\nu}} + 5\overline{\nu}o_N(1)\right)\frac{1}{f_l}(N+3\ell_u)^2 +2\left(\frac{2}{\overline{\nu}}+o_N(1)\right)ja_iN^2.
\]
Similarly, by suitably  bounding the sequence $(\tau_k)_{k \in \N}$ from below stochastically, one can show that
\[
\liminf_{m\to\infty}\frac{1}{m}\sum_{n=0}^{T_m} 1_{\mathbb{W}_l}(W_n) \geq \left(\frac{\nu_l}{\nu} - \frac{10}{\overline{\nu}} -5\overline{\nu}o_N(1)\right)\frac{1}{f_l}(N-3\ell_u)^2.
\]


\smallskip

\noindent {\it Step 5.}
From Step 4 we conclude that $P_{0,\omega_{\n,\p}}$-almost surely, for each $l=1,\dots,j$,
\[
\limsup_{m\to\infty}\frac{1}{m}\sum_{n=0}^m 1_{S_l}(X_n)
\le 
\frac{\left(\frac{\nu_l}{\nu} + \frac{10}{\overline{\nu}} +5\overline{\nu}o_N(1)\right)\frac{1}{f_l}(N+3\ell_u)^2 + 2\left(\frac{2}{\overline{\nu}}+o_N(1)\right)ja_iN^2}{\sum_{i=1}^j\left(\frac{\nu_i}{\nu} - \frac{10}{\overline{\nu}} -5\overline{\nu}o_N(1)\right)\frac{1}{f_i}(N-3\ell_u)^2}.
\]
Hence, by setting $\nu_l:=\lfloor M \lambda_l f_l\rfloor$ for each $l=1,\dots,j$, then choosing $M$ large enough and finally taking $N$
 sufficiently large, we obtain the upper bound corresponding to \eqref{emp}.
 The lower bound can be derived similarly.
\end{proof}


\begin{proposition}
  \label{exrwpe}
  If $\bP$ is non-nestling or marginally nestling then, given $\varepsilon > 0$, there exists a period $(\n,\p)$ such that  $\p=(\vec{p}(x))_{x\in B_{\bf n}}$ is such that for all $x\in B_{\bf n}$, $p(x)\in\mathcal S(\mathbb P)$ and the corresponding $(\n,\p)$-RWPE satisfies
\[
\left|\limsup_{n\to\infty}\frac{1}{n}\log P_{0,\omega_{\n,\p}} (X_n=0)-\left(-\inf_{\sigma\in 
    {\mathcal K}_{\bP}}I_{\sigma}(0)\right)\right|\le\varepsilon. 
\]
\end{proposition}

 \begin{proof} Again, we subdivide the proof into various steps.
   
\noindent {\it Step 0}. We begin by proving that, for any periodic environment $\omega_{\n,\p}$ with support $\cS$, we have the upper bound
\begin{equation}\label{eq:upbound}
\limsup_{n \to \infty} \frac{1}{n} \log P_{0,\omega_{\n,\p}}(X_n=0)\le  -\inf_{\sigma\in\mathcal K_{\cS}}I_\sigma(0).
\end{equation} 
To this end, note that, for
   any $\theta\in\mathbb R^d$, the exponential Chebychev inequality together with the Markov property supplies us with the upper bound
\[
P_{0,\omega_{\n,\p}}(X_n=0)\le E_{0,\omega_{\n,\p}}\big[ \mathrm{e}^{\langle \theta, X_n \rangle }\big]\le \Big(\sup_{\sigma\in\mathcal K_{\cS}}\mathrm{e}^{\Lambda_\sigma(\theta)}\Big)^n=\mathrm{e}^{n \sup_{\sigma \in \cK_\cS} \Lambda_\sigma(\theta)}.
\] 
As a consequence,
\[
\limsup_{n \to \infty} \frac{1}{n}P_{0,\omega_{\n,\p}}(X_n=0)\le \inf_{\theta\in\mathbb R^d}\sup_{\sigma\in\mathcal K_{\cS}}\Lambda_\sigma(\theta)
=-\inf_{\sigma\in\mathcal K_{\cS}}I_\sigma(0),
\] where in the last equality we have again used Sion's minimax theorem as in Lemma~\ref{lem:saddle}. Hence, we obtain \eqref{eq:upbound}.

\noindent {\it Step 1.} Let $\varepsilon>0$ and,  recalling from above \eqref{eq:rateFunc} that $\cS(\bP)$ denotes the essential support of $\omega(0)$ under $\bP$, choose
   \begin{equation}
   \label{sss}
   \cS:=\{\sigma_1,\ldots,\sigma_j\}\subseteq \Big\{ \sigma \in \cP_{\bV} : \exists\,\vec{p} \in \cS(\bP) \text{ with }\sup_{e \in \bV} |\sigma(e)-p(e)| < \varepsilon\Big\}
   \end{equation}
   such that 
\begin{equation}\label{eq:approx}
\Big|\inf_{\sigma \in \cK_\cS} I_{\sigma}(0)-\inf_{\sigma\in\mathcal K_{\mathbb{P}}} I_\sigma(0)\Big|\le\varepsilon.
\end{equation}  
Since the infimum in $\inf_{{\mathcal K}_\bP}I_\sigma(0)$ is
attained in $\partial{\mathcal K}_\bP$ (see Lemma~\ref{lem:boundary} below), we can assume
without loss of generality that  the convex hull of the set $\mathcal{D}_\cS:=\{ \rmd(\sigma) : \sigma \in \cS\}$ is a  subset of $\mathbb R^d$ of dimension smaller than or equal to $d-1$ and thus we may take $j \leq d$.
By Lemma~\ref{lem:saddle}, there exist $t^*\in \Delta_j$ and $\theta^* \in \bR^d$ such that $(t^*,\theta^*)$ is a saddle point of the mapping $\Lambda$ from Lemma~\ref{lem:saddle} so that, in particular, we have
\[
\Lambda(t^*,\theta^*)=\inf_{\theta \in \bR^d}\max_{t \in \Delta_j}\Lambda(t,\theta)
=\max_{t \in \Delta_j} \inf_{\theta \in \bR^d}\Lambda(t,\theta)=-\inf_{\sigma \in \cK_\cS} I_\sigma(0).
\] Without loss of generality, we may assume that $t^*=(t_1^*,\dots,t_j^*)$ satisfies $t_i^*> 0$ for all~$i$, as otherwise we may replace $\cS$ by $\cS^*:=\{ \sigma_i : t_i^*>0\}$ which, by definition of $\Lambda$ and the fact that $(t^*,\theta^*)$ is a saddle point of $\Lambda$, verifies that $\inf_{\sigma \in \cK_{\cS^*}} I_\sigma(0)=\inf_{\sigma \in \cK_\cS} I_\sigma(0)$. In particular, under this assumption, it follows from the proof of Proposition~\ref{prop31} that $\Lambda_{\sigma_i}(\theta^*)=:\overline{\Lambda}(\theta^*)$ does not depend on the choice of $i=1,\dots,j$.
Therefore, for $i=1,\dots,j$, we may consider the tilted environment $\sigma_i^*$, defined for each $e \in \bV$ via the formula
\[
\sigma^*_{i}(e):=\mathrm{e}^{\langle \theta^*,e\rangle - \Lambda_{\sigma_i}(\theta^*)}\sigma_i(e)=\mathrm{e}^{\langle \theta^*,e\rangle-\overline{\Lambda}(\theta^*)}\sigma_i(e).
\]
We now observe that 
\begin{equation}
\label{tdst}
\sum_{i=1}^j t_i^* \mathrm{d}(\sigma_i^*)=0.
\end{equation}
Indeed, since $\theta^*$ is a global minimum of $\Lambda_{\langle t^*,\sigma \rangle}$, where $\Lambda_{\langle t^*,\sigma \rangle}$ is defined as in \eqref{eq:lgf1} for $\langle t^*,\sigma\rangle:=\sum_{i=1}^j t_i^* \sigma_i$, we have in particular that $\nabla_\theta \Lambda_{\langle t^*,\sigma \rangle} (\theta^*)=0$, from where the fact that $\sum_{i=1}^j t_i^* \mathrm{d}(\sigma_i^*)=0$ immediately follows.
Since the convex hull of $\mathcal{D}_{\cS}$ is a subset of $\mathbb R^d$ of dimension smaller than or equal to $d-1$,
there must exist some $u\ne 0$ such that
\[
\langle \rmd(\sigma_i^*),u \rangle = 0
\]
for all $1\le i\le j$. Furthermore, by choosing the $\sigma_i$ appropriately, we can guarantee that $u$ has rational coordinates. We can now apply Lemma \ref{l32} to find a random walk $(X_n)_{n \in \bN_0}$ in some periodic
environment $\omega^*_{\n,\p}$ supported on $\sigma_1^*,\dots,\sigma_j^*$,  
whose empirical measure satisfies
for each $1\le k\le j$,
\begin{equation}
\label{lk}
\limsup_{n\to\infty}\left|\frac{\sum_{i=0}^n 1_{S_k}(X_i)}{n}-t_k^*\right|\le \frac{\varepsilon}{4},
\end{equation}
where $S_k$ is the
set of sites $x\in \bZ^d$ such that $\omega^*_{\n,\p}(x)=\sigma^*_k$.\\

   \noindent {\it Step 2.} Next, we give a construction of the RWPE from Step 1 which will be
   convenient for the proof of the lower bound. For each
   $1\le i\le j$, let $(Z^{(i)}_n)_{n \in \bN_0}$
   be an independent random walk jumping according to $\sigma^*_i$. Set $X_0=0$, $\tau^{(i)}_0:=1_{S_i}(X_0)$ and, for each~$n \in \bN$, define recursively
\[
X_{n}:=\sum_{i=1}^j Z^{(i)}_{\tau^{(i)}_{n-1}} \qquad \text{ and }\qquad \tau^{(i)}_n:=\sum_{k=0}^n 1_{S_i}(X_k).
\]
Clearly $(X_n)_{n \in \bN_0}$ has the law of a random walk in the space-periodic
environment $\omega^*_{\n,\p}$ from Step 1.\\

\noindent {\it Step 3.} Finally, consider the space-periodic environment $\omega_{\n,\p}$ defined on each $x \in \bZ^d$ via the condition $\omega_{\n,\p}(x)=\sigma_i(x)$ if and only if $\omega_{\n,\p}^*(x)=\sigma_i^*$, i.e., 
\[
\omega^*_{\n,\p}(x,e)=\mathrm{e}^{\langle \theta^*,e\rangle-\overline{\Lambda}(\theta^*)}\omega_{\n,\p}(x,e).
\] 
Note that the Radon-Nikodym derivative of the law $P_{0,\omega_{\n,\p}^*}$ up to time $n$ with
respect to $P_{0,\omega_{\n,\p}}$ is 
$\mathrm{e}^{\langle \theta^*, X_n\rangle -n\overline{\Lambda}(\theta^*)}$. 
Now, for each $n \in \bN$ consider the events
\[
A^{(\varepsilon)}_n:=\bigcap_{i=1}^j\left\{|Z^{(i)}_{nt^*_i}-nt^*_i \rmd(\sigma_i^*)|\le \frac{n\varepsilon}{2j} \right\}
\]
and
\[
B^{(\varepsilon)}_n:=\bigcap_{i=1}^j\left\{|\tau^{(i)}_{n-1}-nt_i^* |\le \frac{n\varepsilon}{2j} \right\}.
\]
We will show that there exists a constant $c > 0$
such that
\begin{equation}
  \label{lb}
  P_{0,\omega_{\n,\p}}(|X_n|\le n\varepsilon)\ge
\mathrm{e}^{-cn\varepsilon + n\overline{\Lambda}(\theta^*)}P_{0,\omega^*_{\n,\p}}\big(
  A^{(\varepsilon)}_n \cap B^{(\varepsilon)}_n \big).
\end{equation}
  From (\ref{tdst}), it is straightforward to check that the inclusion $A^{(\varepsilon)}_n \cap B^{(\varepsilon)}_n \subseteq \{ |X_n| \leq n\varepsilon\}$ holds for all $n \in \bN$. On the other hand, by the definition of $\omega_{\n,\p}$, we have that
\[ E_{0,\omega_{\n,\p}}\big[\mathrm{e}^{\langle \theta^*,X_n\rangle}; |X_n|\le n\varepsilon \big]
 = \mathrm{e}^{n\overline{\Lambda}(\theta^*)} P_{0,\omega^*_{\n,\p}}\left( |X_n|\le n\varepsilon \right).
\]
Hence, by the Cauchy-Schwarz inequality, for some constant $c>0$ (depending only on the dimension $d$ and $\theta^*$) we have that
\[
\mathrm{e}^{n\overline{\Lambda}(\theta^*)} P_{0,\omega^*_{\n,\p}}\left( A^{(\varepsilon)}_n \cap B^{(\varepsilon)}_n\right)\le
\mathrm{e}^{n\overline{\Lambda}(\theta^*)} P_{0,\omega^*_{\n,\p}}\left(|X_n|\le n\varepsilon \right)
\le \mathrm{e}^{cn\varepsilon }P_{0,\omega_{\n,\p}}(|X_n|\le n\varepsilon).
\] This immediately gives \eqref{lb}.\\

\noindent {\it Step 4.} From the law of large numbers and \eqref{lk} in Step 1, for every $i=1,\dots,j$,
\[
\limsup_{n\to\infty}\Big|\frac{\tau^{(i)}_{n-1}}{n}-t_i^*\Big|\leq \frac{\varepsilon}{4} \qquad P_{0,\omega^*_{\n,\p}}-a.s.  
\]
and
\[
\lim_{n\to\infty}\frac{Z^{(i)}_{n}}{n}= \rmd(\sigma_i^*)\qquad P_{0,\omega^*_{\n,\p}}-a.s.
\]
It follows that $\lim_{n\to\infty} P_{0,\omega^*_{\n,\p}}(A^{(\varepsilon)}_n \cap B^{(\varepsilon)}_n)=1$.
Therefore, from Step 3, we see that for every $\varepsilon>0$,
\begin{equation}
\label{fff}
\limsup_{n\to\infty}\frac{1}{n}\log P_{0,\omega_{\n,\p}}(|X_n|\le n\varepsilon)\ge -c\varepsilon+\overline{\Lambda}(\theta^*) =-c\varepsilon-\inf_{\sigma \in \cK_\cS} I_\sigma(0),
\end{equation} where the last inequality follows from the proof of Proposition~\ref{prop31}.
Now, note that for each $\varepsilon$, the
random walk in the periodic environment with law $P_{0,\omega_{\n,\p}}$
satisfies a large deviation principle with rate functions
$I_{\bf n,\bf p}$ which are convex and continuous in $D^\circ$ (see the explanation
given after display (\ref{eq:rateFunc})). Now, the convexity and continuity of these
rate functions 
in $D^o$ (see \cite{V03}), together with the fact that they are bounded by  $|\log\kappa|$ there (where $\kappa$ is the ellipticity constant), imply that they are equicontinuous in any
compact subset of $D^\circ$. Therefore we have that
\[
\limsup_{n\to\infty}\frac{1}{n}\log P_{0,\omega_{\n,\p}}(X_n=0)
\ge \limsup_{n\to\infty}\frac{1}{n}\log P_{0,\omega_{\n,\p}}(X_n\le n\varepsilon)-g(\varepsilon),
\]
where $g(\varepsilon)=o(\varepsilon)$. Combining this with (\ref{fff}),
we conclude that for every $\varepsilon>0$, there exists
a periodic environment $\omega_{\bf n,\bf p}$ with $\vec{p}(x)\in\mathcal S$ for all $x\in B_{\bf n}$ such that

\begin{equation}
\label{llll}
\limsup_{n\to\infty}\frac{1}{n}\log P_{0,\omega_{\n,\p}}(X_n=0)\ge -c\varepsilon-g(\varepsilon)-\inf_{\sigma \in \cK_\cS} I_\sigma(0),
\end{equation}

On the other hand, by the definition
of $\mathcal S$ in (\ref{sss}) Step 1, we know that the exists a periodic environment $\omega_{\bf n,\bf p'}$ such that for every $x\in B_{\bf n}$ and $e\in\mathbb V$, $|p(x,e)-p'(x,e)|\le\varepsilon$, and $\vec{p}\, '\in\mathcal S(\mathbb P)$. Hence,
from (\ref{llll}) we conclude that

\begin{equation}
\nonumber
\limsup_{n\to\infty}\frac{1}{n}\log P_{0,\omega_{\n,\p'}}(X_n=0)\ge -c\varepsilon-g(\varepsilon)-\inf_{\sigma \in \cK_\cS} I_\sigma(0),
\end{equation}
   This concludes the proof of
  the lower bound. The upper bound follows from Step 0 since $\cS$ satisfies \eqref{eq:approx}.
\end{proof}

\section{Proof of Theorem \ref{theorem1}} \label{sec:proofMain}

The proof of part $(i)$ of Theorem \ref{theorem1}  will be separated into two parts, one on the upper bound  and another on the lower bound for the large deviation probabilities. The proof of upper bound is straightforward, while proving the lower bound takes much effort and tells us more about the behavior of quenched RWRE.
\subsection{Part $(i)$ of Theorem \ref{theorem1} (upper Bound of (\ref{2p1}))}
Due to the continuity of $I$ in the origin (cf.\ Section \ref{sec:NotRes}), in order to establish $I(0)\ge -\log \left(\inf_{\theta\in\mathbb R^d}\sup_{\sigma\in\mathcal K_{\mathbb P}}
  \sum_{e\in \bV}e^{\langle \theta, e\rangle }\sigma(e)\right)$,
it is sufficient to prove that 
 for $\mathbb{P}$-almost all $\omega$ we have
\begin{equation} \label{eq:0ProbBd}
\limsup_{n \rightarrow \infty } \frac{1}{n}\log P_{0,\omega}(X_n=0) \leq \log \Big(\inf_{\theta\in\mathbb R^d}\sup_{\sigma\in\mathcal K_{\mathbb P}}
  \sum_{e\in \bV}e^{\langle \theta, e\rangle}\sigma(e)\Big).
\end{equation}
Note that for all $\theta \in \R^d$ and $\mathbb{P}$-almost all $\omega \in \Omega$, we have 
 \begin{eqnarray*}
 & P_{0,\omega}(X_n=0) \leq E_{0,\omega}[e^{\langle \theta, X_n\rangle}]\leq \Big(\sup_{\sigma\in\mathcal K_{\mathbb P}}
 \sum_{e\in U}e^{\langle \theta, e\rangle}\sigma(e)\Big)^n
 \end{eqnarray*}
 for all $n\ge 0$, so
 $$
\limsup_{n\to\infty}\frac{1}{n}\log P_{0,\omega}
(X_n=0)\le \log
\Big(\sup_{\sigma\in\mathcal K_{\mathbb P}}
 \sum_{e\in U}e^{\langle \theta, e\rangle }\sigma(e)\Big).
 $$
 Now \eqref{eq:0ProbBd} follows by taking the infimum over $\theta$.

\subsection{Proof of part $(i)$ (lower bound of (\ref{2p1})), parts $(ii)$ and $(iii)$}
First note that it is always possible to choose $\delta_0$ small enough so that (\ref{bccl}) is satisfied.

Next we will prove that 
$$
I(0)\le -\log \Big(\inf_{\theta\in\mathbb R^d}\sup_{\sigma\in\mathcal K_{\mathbb P}}
  \sum_{e\in U}e^{\langle\theta, e\rangle }\sigma(e)\Big).
  $$
For this purpose, we provide a sufficiently likely strategy for the RWRE to be in the origin at large times $n$. More precisely, we will force the random walk to move to the
center of a ball of radius $\delta(\log n)^{1/d}$  inside a box of side length
$\frac{n}{\log n}$, where the environment is periodic and generates
the appropriate large deviation rate function, in time $O(\frac{n}{\log
    n})$, then stay in this ball a time of order $n$, and
move back to $0$, again in time  $O(\frac{n}{\log n})$.

\smallskip

\begin{lemma} We have that for $\mathbb P$-almost all $\omega \in \Omega$,
  $$
    \liminf_{m \rightarrow \infty }\frac{1}{2m}\log 
 P_{0,\omega}(X_{2m}=0)
  \geq \log \Big(\inf_{\theta \in \R^d} \sup_{\sigma \in \mathcal
    K_{\mathbb P}} \sum_{e \in U}
  e^{\langle \theta,  e\rangle}\sigma(e)\Big).
  $$
\end{lemma}
\begin{proof}
 As $\log \sum_e e^{\langle \theta , e\rangle}\sigma(e)$ is convex in $\theta$ and concave in $\sigma(\cdot)$, using a minimax argument, we observe that
 $$\inf_{\theta \in \R^d} \sup_{\sigma \in \mathcal K_{\mathbb P}} \sum_e e^{\langle\theta ,
   e\rangle}\sigma(e)=\sup_{\sigma \in \mathcal K_{\mathbb P}} \inf_{\theta \in \R^d} \sum_e e^{\langle \theta
   , e\rangle}\sigma(e).
 $$
 Consider now $\sigma$ that maximizes $\inf_{\theta \in 
   \R^d} \sum_e e^{\langle \theta , e\rangle}\sigma(e)$.
 Using \eqref{eq:rateFunc} and  Proposition \ref{exrwpe}, for each $\varepsilon > 0$
 there exists an $(\mathbf n,\mathbf p)$-RWPE such that
 \begin{equation}
 \label{penv}
 \lim_{n\to\infty}\frac{1}{n}\log P_{0,\omega_{\bf n, \bf p}}(X_n=0)=-I_{\bf n,\bf p}(0)\ge 
 \log\Big(\inf_{\theta \in 
   \R^d} \sup_{\sigma\in \mathcal K_{\mathbb P}}\sum_e e^{\langle \theta , e\rangle }\sigma(e)\Big) - h(\varepsilon),
 \end{equation}
 where $h(\varepsilon)=o(\varepsilon)$.
 (\ref{bccl}) implies that
 in the box $\rmB(0,\frac{n}{\log n})$ (recall the definition below \eqref{def:GDef}) 
for each $\varepsilon >0$ there exists $\delta > 0$ such that
with probability $1$, eventually in $n$  there exists $z_n \in \rmB(0,\frac{n}{\log n})$  such that for all
 $x\in  B_n:=\{x'\in\mathbb Z^d:|x'-z_n|_\infty \le \delta (\log n)^\frac1d\}$ we have 
 \begin{equation} \label{eq:omegaApprox}
 |\omega(x,e)-\omega_{\mathbf n,\mathbf p}(x,e)|\le\varepsilon.
 \end{equation}
Let $a_n=o(n)$. Now, consider the following events: 
\begin{enumerate}[(1)]
    \item  $A_1$, the event that the
random walk $(X_n)_{n\ge 0}$ moves from $0$ to the center $z_n$ of the ball
$B_n$ in $a_n$  steps; 
\item $A_2$ the event that the random walk
$(X_n)_{n\ge 0}$ stays between times $0$ and $n-2a_n$ at $\infty$-distance at most $\delta (\log n)^\frac1d$ to $z_n$, and that at time $n-2a_n$ it ends up in the center $z_n$ of the
ball $B_n$; 
\item $A_3$, the event that the random walk $(X_n)_{n\ge 0}$
moves in $a_n$ steps, from the center $z_n$
of the ball $B_n$ to the center of the box $C_n$. 
\end{enumerate}
By the Markov
property, we then have, for appropriate choices of $a_n$, and $n$ even, that
$$
P_{0,\omega}(X_n=0)\ge
P_{0,\omega}(A_1)P_{z_n,\omega}(A_2)P_{z_n,\omega}(A_3)\ge \kappa^{2a_n}P_{z_n,\omega}(A_2),
$$
where we have used the uniform ellipticity property and the fact that
the side of the box $C_n$ is of length $o(n)$.

We will now lower bound the probability $P_{z_n,\omega}(A_2)$. Indeed, bearing in mind \eqref{eq:omegaApprox} we note that
\begin{equation}
\label{ll1}
P_{z_n,\omega}(A_2)\ge e^{-c\varepsilon n}P_{z_n, \omega_{\mathbf n, \mathbf p}}
\Big(X_{n-2a_n}=z_n, \sup_{0\le m\le
  n-2a_n}|X_m-z_n|_\infty\le \delta (\log n)^{1/d}\Big).
\end{equation}
Now note that for $n$ even, we have that
$$
P_{z_n, \omega_{\mathbf n, \mathbf p}} (X_{n-2a_n}=z_n)>0.
$$
Furthermore,
writing $\theta_*$ for the argmin of the Laplace transform $\Lambda_{X_1}$ of $X_1$, so that 
$$
E_{z_n, \omega_{\mathbf n, \mathbf p}} \big[e^{\langle \theta_* , X_1\rangle }\big]=\min_{\theta\in\mathbb R^d}E_{z_n, \omega_{\mathbf n, \mathbf p}} \big[e^{\langle \theta , X_1\rangle}\big],
$$
we denote by
$P_*$ the law under which the random walk $X_1$ jumps from $x$ to $x\pm e_i$ with probability
$\frac{e^{\theta_{*,i}}\omega_{\mathbf n, \mathbf
    p}(x,e_i)}{\Lambda_{X_1}(\theta_*)}$ and
  $\frac{e^{-\theta_{*,i}}\omega_{\mathbf n,\mathbf p}(x,-e_i)}{\Lambda_{X_1}(\theta_*)}$, respectively. Writing furthermore
$I_{\bf n,\bf p}$ for the rate function of RWPE in the environment
$\omega_{\mathbf n, \mathbf p}$, we obtain
\begin{align}\label{decomp}
\begin{split}
&P_{z_n, \omega_{\mathbf n, \mathbf p}} \Big(X_{n-2a_n}=z_n, \sup_{0 \leq m \leq n-2a_n}|X_m-z_n|_\infty \leq \delta (\log n)^{1/d}\Big)\\
&=\Lambda^{n-2a_n}_X(\theta_*)P_*\Big(X_{n-2a_n}=0, \sup_{0 \leq m \leq n-2a_n}|X_m|_\infty \leq \delta  (\log n)^{1/d}\Big)\\
&=e^{-(n-2a_n)I_{\bf n,\bf p}(0)} P_*\Big (X_{n-2a_n}=0, \sup_{0 \leq m \leq n-2a_n}|X_m|_\infty \leq \delta (\log n)^{1/d}\Big ).
\end{split}
\end{align}
From (\ref{ll1}) and \eqref{decomp}, we see that it is enough to prove that 
$$
\limsup_{n \rightarrow \infty} \frac{1}{n} \log P_* \Big (\sup_{0 \leq m \leq n-2a_n}|X_m|_\infty \leq \delta (\log n)^{1/d}, X_{n-2a_n}=0\Big)=0.
$$
To do this, we will consider paths of $(X_n)_{n\ge 0}$ which return to
$0$ every  $L_n=2\lfloor \delta (\log n)^{1/d}\rfloor$ steps. Without loss of
generality, we will choose $a_n$ so that  $L_n$ is a divisor of
$n-2a_n$. Hence,

$$
P_*\Big(\sup_{0 \leq m \leq n-2a_n}|X_m| \leq \delta (\log n)^{1/d}, 
X_{n-2a_n}=0\Big)
\ge \big( P_*(X_{L_n}=0)\big)^{(n-2a_n)/L_n}.
$$
Next, we will apply the local limit theorem for (centered) RWPE, Theorem 1.1, proved by Takenami in \cite{T00}. While this is proved for aperiodic RWPE, it can be easily adapted to
the nearest neighbor case using the fact that the even sublattice of $\mathbb Z^d$ is isomorphic to $\mathbb Z^d$. While the remaining assumptions of that source are easy to check in our setting, let us note that the second part of Assumption 1.3 of \cite[Theorem 1.1]{T00} (centered increments under stationary distribution) follows from the above transition from $P_{z_n, \omega_{\mathbf n, \mathbf p}}$ to $P_*$. 
As a consequence, 
\cite[Theorem 1.1]{T00} then in particular entails the 
existence of a constant $C>0$ such that
$$
P_*(X_{L_n}=0)=\frac{C}{( L_n)^{d/2}}+o\left(L_n^{-d/2}\right).
$$
Hence, for $n$ even,
$$
P_*\Big (\sup_{0 \leq m \leq n-a_n}|X_m|_\infty\leq \delta (\log n)^{\frac{1}{d}}, 
X_{n-a_n}=0 \Big) \ge  Ce^{-C\frac{n\log\log n}{(\log n)^{1/d}}}.
$$
Finally, we construct a deterministic path from the center of the
neighborhood to the origin, which happens with probability at least
$\kappa^{a_n}$. Putting this together, we conclude that for $n$ even,
$$
P_{0,\omega}(X_{n}=0) \geq C\kappa^{2a_n} e^{-(n-a_n)I_{\bf n,\bf p}(0)}
  e^{-C\frac{n\log\log n}{(\log n)^{1/d}}}.
$$
Thus, by (\ref{penv}),
$$
\liminf_{m\to\infty}\frac{1}{2m}
\log P_{0,\omega}(X_{2m}=0)
\geq -I_{\bf n,\bf p}(0)\ge\log \Big(\inf_{\theta \in \R^d} \sup_{\sigma\in \mathcal K_{\mathbb P}} \sum_e
e^{\langle \theta , e\rangle }\sigma(e)\Big)-h(\varepsilon).
$$
\end{proof}

\subsection{Proof of part (i) (maximizer at the boundary of the convex hull of the law)}

\begin{lemma}\label{lem:boundary}
 In the non-nestling or marginally nestling cases, the supremum
 \[
 \sup_{\vec{p} \in \mathcal{K}_{\mathbb P}} \inf_{\theta \in \R^d} \sum_{e \in \bV} \mathrm{e}^{\langle \theta , e\rangle }p(e)
 \] is attained on the boundary of $\mathcal{K}_{\mathbb P}$.    
 \end{lemma}
\begin{proof}
 We know that $\inf_{\theta \in \R^d} \sum_{e \in \bV} \mathrm{e}^{\langle \theta , e\rangle }p(e)=2\sum_{i=1}^d \sqrt{p_iq_i},$ where $p_i:=p(e_i)$ and $q_i:=p(-e_i)$. Thus, it is enough to study $\sup_{p \in \mathcal{K}_{\mathbb P}} \sum_{i=1}^d \sqrt{p_iq_i}$ subject to the constraint $\sum_{i=1}^d (p_i+q_i)=1$ and $p_i,q_i > 0$, as the environment is uniformly elliptic. Applying the method of Lagrange multipliers, we start by 
 defining the Lagrangian function of the $p_i$, $q_i$ and $\lambda\in\mathbb R$,
 \[
 L(p,\lambda):=\sum_{i=1}^d \sqrt{p_iq_i}+\lambda \Big(\sum_{i=1}^d (p_i+q_i)-1\Big).
 \]
 Thus,
 $$\frac{\partial L}{\partial p_i}=\frac{1}{2} \sqrt{\frac{q_i}{p_i}}+\lambda=0$$
 $$\frac{\partial L}{\partial q_i}=\frac{1}{2} \sqrt{\frac{p_i}{q_i}}+\lambda=0$$
 $$\sum_{i=1}^d (p_i+q_i)=1$$
 Solving these equations, we obtain that $\lambda=-\frac{1}{2}$ and $p_i=q_i$ for all $i$. Hence, the only local maximum of $\sum_{i=1}^d \sqrt{p_iq_i}$ corresponds to the zero drift case. Now, suppose that $\sup_{\vec{p}\in \mathcal{K}_{\mathbb P}} \inf_{\theta \in \R^d} \sum_{e \in \bV} \mathrm{e}^{\langle \theta , e\rangle }p(e)$ is achieved at some $\vec{p}_0 \in \textrm{int}(\mathcal{K}_{\mathbb P})$.
 In the non-nestling case, this would imply that
  $\sum_{i=1}^d \sqrt{p_iq_i}$ would have a local maximum at $\vec{p}_0$ where $\vec{p}_0$ has a non-zero drift, leading to a contradiction. On the other hand, in the marginally nestling case, we know that there exists some $\vec{p}_1\in\partial \mathcal K_{\mathbb P}$, such that the $\rmd(\vec{p}_1)=0$. But it is easy to see that for this $\vec{p}_1=(p_{1,1},\ldots,p_{1,d},q_{1,1},\ldots, q_{1,d})$ one has that
\[
  \sum_{i=1}^d\sqrt{p_{1,i}q_{1,i}}=\sum_{i=1}^d p_{1,i} = \sum_{i=1}^d q_{1,i}=\frac{d}{2},
\]
  which is the common value at all points with zero drift. 
  This completes the proof. 
\end{proof}

\end{document}